\documentclass[11pt]{amsart}
\usepackage[centertags]{amsmath}
\usepackage{amsfonts}
\usepackage{amssymb}
\usepackage{amsthm}
\usepackage{newlfont}
\usepackage{amscd}
\usepackage{mathrsfs}
\usepackage[all,cmtip]{xy}
\usepackage{tikz-cd}
\tikzcdset{arrow style=Latin Modern, row sep/normal=1.35em, column sep/normal=1.8em, cramped}
\usepackage{tikz}
\usepackage[pagebackref=false,colorlinks]{hyperref}
\usepackage{array}

\definecolor{mycolor}{HTML}{F7F8E0}
\definecolor{myorange}{RGB}{245,156,74}
\definecolor{cadetgrey}{rgb}{0.57, 0.64, 0.69}
\definecolor{calpolypomonagreen}{rgb}{0.12, 0.3, 0.17}

\hypersetup{pdffitwindow=true,linkcolor=calpolypomonagreen,citecolor=calpolypomonagreen,urlcolor=calpolypomonagreen}
\usepackage{cite}
\usepackage{fancyhdr}
\usepackage{stmaryrd}

\usepackage[OT2,OT1]{fontenc}
\newcommand\cyr{%
\renewcommand\rmdefault{wncyr}%
\renewcommand\sfdefault{wncyss}%
\renewcommand\encodingdefault{OT2}%
\normalfont
\selectfont}
\DeclareTextFontCommand{\textcyr}{\cyr}

\newcommand{\Mod}[1]{\ (\text{mod}\ #1)}

\usepackage[toc, page] {appendix}

\numberwithin{equation}{section}

\newtheorem{thm}{Theorem}[section]

\newtheorem{cor}[thm]{Corollary}
\newtheorem{lem}[thm]{Lemma}
\newtheorem{prop}[thm]{Proposition}

\newtheorem{assu}[thm]{Assumption}

\theoremstyle{definition}
\newtheorem{defn}[thm]{Definition}

\newtheorem{rem}[thm]{Remark}

\newcommand{\ks}{\boldsymbol{\kappa}}
\newcommand{\lambdab}{\boldsymbol{\lambda}}
\newcommand{\kn}{\widetilde{\boldsymbol{\delta}}}
\newcommand{\widedelta}{\widetilde{\delta}}

\begin{document}
\title{A higher Gross--Zagier formula and the structure of Selmer groups}
\author{Chan-Ho Kim}
\address{Center for Mathematical Challenges, Korea Institute for Advanced Study, 85 Hoegiro, Dongdaemun-gu, Seoul 02455, Republic of Korea}
\email{chanho.math@gmail.com}
\date{\today}
\begin{abstract}
We describe a Kolyvagin system-theoretic refinement of Gross--Zagier formula by comparing Heegner point Kolyvagin systems with Kurihara numbers when the root number of a rational elliptic curve $E$ over an imaginary quadratic field $K$ is $-1$.
When the root number of $E$ over $K$ is 1, we first establish the structure theorem of the $p^\infty$-Selmer group of $E$ over $K$.
The description is given by the values of certain families of quaternionic automorphic forms, which is a part of bipartite Euler systems.
By comparing bipartite Euler systems with Kurihara numbers, we also obtain an analogous refinement of Waldspurger formula.
No low analytic rank assumption is imposed in both refinements.

We also prove the equivalence between the non-triviality of various ``Kolyvagin systems" and the corresponding main conjecture localized at the augmentation ideal.
As consequences, we obtain new applications of (weaker versions of) the Heegner point main conjecture and the anticyclotomic main conjecture to the structure of $p^\infty$-Selmer groups of elliptic curves of arbitrary rank. In particular, the Heegner point main conjecture localized at the augmentation ideal implies the strong rank one $p$-converse to the theorem of Gross--Zagier and Kolyvagin.

\end{abstract}
\maketitle

\setcounter{tocdepth}{1}
\tableofcontents
\section{Introduction}
\subsection{A ``higher" Gross--Zagier formula}
The cerebrated Gross--Zagier formula \cite{gross-zagier-original} is a remarkable formula relating the first central derivatives of Rankin--Selberg $L$-functions with the heights of Heegner points, and it has played an important role to study the Birch and Swinnerton-Dyer conjecture.
The Waldspurger formula \cite{waldspurger} also presents a beautiful relation between the Rankin--Selberg $L$-values and certain toric period integrals, which covers the complement to the Gross--Zagier formula in terms of the sign of the functional equation.
These formulas have been enormously extended to
Shimura curves over totally real fields \cite{yuan-zhang-zhang} and various higher dimensional Shimura varieties in the context of the arithmetic Gan--Gross--Prasad conjecture \cite{wei-zhang-arith-fund-lem, wei-zhang-weil-repns-afl} when the rank is at most one (in an appropriate sense). 

When the rank is larger than one, the situation becomes extremely puzzled.
The only known work is essentially  Yun--Zhang's ground-breaking work on the Gross--Zagier formula for the function field case \cite{yun-zhang-higher-gross-zagier, yun-zhang-higher-gross-zagier-2} and its subsequent work.
It seems that no similar result for the number field case is addressed yet.

In this article, we present certain \emph{structural refinements} of  Gross--Zagier formula and Waldspurger formula for elliptic curves over the rationals from the viewpoint of Kolyvagin systems and the structure of Selmer groups ($\S$\ref{subsec:structural-gross-zagier} and $\S$\ref{subsec:structural-waldspurger}).
In particular, no low rank assumption is imposed.
A simple consequence of these structural refinements is stated in Theorems \ref{thm:main-higher-gross-zagier-easy} and \ref{thm:main-higher-waldspurger-easy} as ``higher" Gross--Zagier and Waldspurger formulas.

When the Heegner hypothesis is satisfied, the structural Gross--Zagier formula compares certain modular symbols of an elliptic curve and its quadratic twist with the corresponding Heegner point Kolyvagin system.

When the Heegner hypothesis breaks down, the structural Waldspurger formula compares the same modular symbols with the values of families of quaternionic automorphic forms at CM points.

These comparisons are purely algebraic and strongly inspired by refined Iwasawa theory \cite{kurihara-munster, kurihara-iwasawa-2012}, especially, the structure theorems of Selmer groups via the Kolyvagin system method \cite{kolyvagin-selmer, mazur-rubin-book}. This is completely different from the relative trace formula approach to the arithmetic Gan--Gross--Prasad conjecture.
It is remarkable that our approach fully captures the structure of Selmer groups without the low analytic rank assumption and Kolyvagin \emph{derivatives} of certain Mazur--Tate elements are involved instead of \emph{derivatives} of Rankin--Selberg $L$-functions.
\subsection{From Euler systems to Kolyvagin systems} \label{subsec:intro-from-euler-to-kolyvagin}
In the arithmetic of elliptic curves, the importance of the theory of Euler systems cannot be overestimated.
In particular, two completely different Euler systems are extensively studied to understand Selmer groups of elliptic curves: Kato's Euler systems \cite{kato-euler-systems} and Heegner points \cite{kolyvagin-euler-systems}.
The following table summarizes their differences.
\begin{center}
\begin{tabular}{ |m{0.25\textwidth}|m{0.3\textwidth}|m{0.3\textwidth}| } 
 \hline
  & Kato's Euler systems & Heegner points \\ 
 \hline
field variation & abelian extensions of the rationals (``cyclotomic")& ring class extensions of imaginary quadratic fields (``anticyclotomic") \\ 
 \hline
 local condition at $p$ & the $p$-relaxed condition & the classical condition \\ 
 \hline
bounded Selmer groups & the $p$-strict Selmer groups over the rationals & the Selmer groups over  imaginary quadratic fields \\ 
 \hline
 additional hypothesis &  & the Heegner hypothesis \\
 \hline
\end{tabular}
\end{center}
Although the natures of these two Euler systems are very different, B. Perrin-Riou  \cite{perrin-riou-rational-pts}  formulated a conjecture making the connection between the bottom class of Kato's Euler system and the bottom class of the Heegner point Euler system when the analytic rank is one, and it is recently settled \cite{bertolini-darmon-venerucci, burungale-skinner-tian-wan, kazim-pollack-sasaki}.
Except these bottom classes, we do \emph{not} expect the existence of a more general comparison between Kato's Euler systems and Heegner point Euler systems since their field variations are \emph{disjoint} except the base imaginary quadratic field.
However, we still hope that Kato's \emph{Kolyvagin systems} and Heegner point \emph{Kolyvagin systems} can be compared, and we provide an explicit relation between their divisibilities by using the modular symbol interpretation of Kato's Kolyvagin systems, namely Kurihara numbers.
This is possible because Kurihara numbers determine the structure of Selmer groups over the rationals \cite{kurihara-munster, kurihara-iwasawa-2012, kim-structure-selmer} and the Heegner point Kolyvagin system also determines the structure of Selmer groups over imaginary quadratic fields under certain assumptions \cite{kolyvagin-selmer}. 

On the way to prove the structural Gross--Zagier formula, we also reveal a \emph{deeper} relation between Kato's zeta elements and Heegner points
and it can be viewed as a structural refinement of Perrin-Riou's conjecture (Remark \ref{rem:structural-perrin-riou}).
The following picture summarizes the context.
{\scriptsize
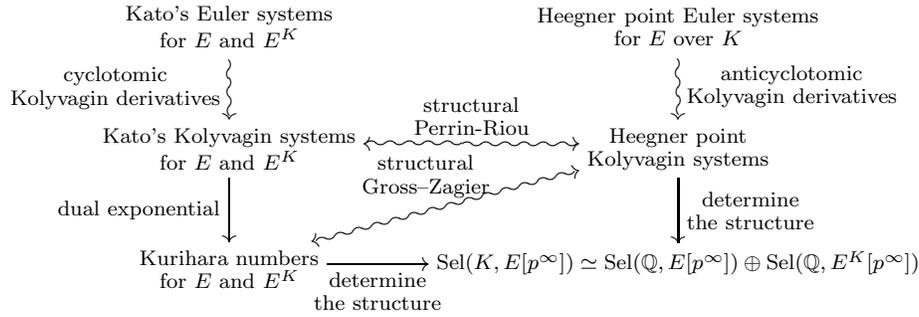
\[
\xymatrix{
\txt{Kato's Euler systems\\for $E$ and $E^K$} \ar@{~>}[d]_-{\txt{cyclotomic \\ Kolyvagin derivatives}} & \txt{Heegner point Euler systems\\for $E$ over $K$} \ar@{~>}[d]^-{\txt{anticyclotomic \\ Kolyvagin derivatives}} \\
\txt{Kato's Kolyvagin systems\\for $E$ and $E^K$} \ar[d]_-{\txt{dual exponential}} \ar@{<~>}[r]^-{\txt{structural \\ Perrin-Riou}} & \txt{Heegner point \\ Kolyvagin systems} \ar[d]^-{\txt{determine \\ the structure}} \ar@{<~>}[dl]_-{\txt{structural \\ Gross--Zagier}} \\
\txt{Kurihara numbers\\for $E$ and $E^K$} \ar[r]_-{\txt{determine \\ the structure}} & 
{ \mathrm{Sel}(K, E[p^\infty])  \simeq  \mathrm{Sel}(\mathbb{Q}, E[p^\infty]) \oplus  \mathrm{Sel}(\mathbb{Q}, E^K[p^\infty]) }
}
\]
}
where $p \geq 5$ is a prime, $E$  is a rational elliptic curve, $K$ is an imaginary quadratic field, and $E^K$ is the quadratic twist of $E$ by $K$ satisfying the (generalized) Heegner hypothesis.
\subsection{Bipartite Euler systems and the structure of Selmer groups} \label{subsec:intro-bipartite}
When the Heegner hypothesis is not satisfied, the Heegner point Kolyvagin system method is certainly not available.
In the extremely inspiring paper \cite{bertolini-darmon-imc-2005}, Bertolini--Darmon developed a completely different Euler system argument involving the systematic use of rank lowering congruences and proved one-sided divisibility of the anticyclotomic Iwasawa main conjecture for elliptic curves.
The theory of bipartite Euler systems \cite{howard-bipartite} captures the essence of this Euler system argument axiomatically.
We push the bipartite Euler system method further to obtain the description of the structure of Selmer groups over imaginary quadratic fields in terms of the values of level-raising families of quaternionic automorphic forms (Theorem \ref{thm:structure-bipartite}). This is the counterpart of Kolyvagin's work on the structure of Selmer groups mentioned above.
As a consequence,  we obtain the structural Waldspurger formula as in the following picture again
{\scriptsize
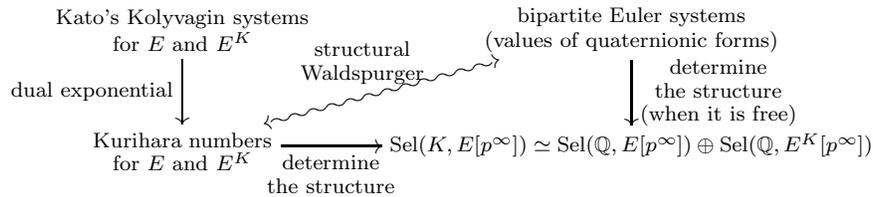
\[
\xymatrix{
\txt{Kato's Kolyvagin systems\\for $E$ and $E^K$} \ar[d]_-{\txt{dual exponential}}  & \txt{bipartite Euler systems \\ (values of quaternionic forms)} \ar[d]^-{\txt{determine \\ the structure \\ (when it is free)}} \ar@{<~>}[dl]_-{\txt{structural \\ Waldspurger}} \\
\txt{Kurihara numbers\\for $E$ and $E^K$} \ar[r]_-{\txt{determine \\ the structure}} & 
{ \mathrm{Sel}(K, E[p^\infty])  \simeq  \mathrm{Sel}(\mathbb{Q}, E[p^\infty]) \oplus  \mathrm{Sel}(\mathbb{Q}, E^K[p^\infty]) }
}
\]
}
with the same notation above but with the opposite root number.
In this sense, we confirm that bipartite Euler systems behave like Kolyvagin systems rather than Euler systems, as pointed out in \cite[Introduction]{bertolini-darmon-imc-2005} and \cite[$\S$1]{howard-bipartite}.
\subsection{The non-triviality questions on various ``Kolyvagin systems"}
Although various structure theorems are discussed above, the non-triviality of the corresponding Kolyvagin system is \emph{required} for each structure theorem.
When Kolyvagin proved his structure theorem \cite{kolyvagin-selmer}, he assumed the non-triviality of Heegner point Kolyvagin systems, and formulated it as a conjecture.
Later, Wei Zhang proved this conjecture for a large class of elliptic curves \cite{wei-zhang-mazur-tate} by using the Iwasawa main conjecture resolved by Skinner--Urban \cite{skinner-urban}.

The non-triviality of Kato's Kolyvagin systems and the collection of Kurihara numbers are also required for Mazur--Rubin's structure theorem \cite{mazur-rubin-book} and Kurihara's structure theorem \cite{kurihara-munster, kurihara-iwasawa-2012, kim-structure-selmer}, respectively.
In \cite{kim-structure-selmer}, we found a method to establish the non-triviality of Kato's Kolyvagin systems and the collection of Kurihara numbers by making connection with the Iwasawa main conjecture.

By extending this idea, we show that the non-triviality of Heegner point Kolyvagin systems and bipartite Euler systems follows from a \emph{small piece of} the corresponding main conjectures, respectively (Theorem \ref{thm:non-triviality-total}).
As a result, we have a plenty of examples satisfying all the required non-triviality statements.
In particular, we obtain a more direct and simplified proof of the Heegner point main conjecture \cite{burungale-castella-kim}, and we refine the standard argument to deduce the strong rank one $p$-converse to the theorem of Gross--Zagier and Kolyvagin from the Heegner point main conjecture (Corollary \ref{cor:non-triviality-heegner}).

\subsection*{Acknowledgement}
The author thanks to K\^{a}z{\i}m B\"{u}y\"{u}kboduk and Naomi Sweeting for helpful discussions and Francesc Castella for pointing out a gap in an earlier version. 
The encouraging discussion with Giada Grossi was very helpful.
We would also like to heartily thank for anonymous referee for a very careful reading of the paper and pointing out important mistakes in earlier versions.

This research was partially supported 
by a KIAS Individual Grant (SP054103) via the Center for Mathematical Challenges at Korea Institute for Advanced Study and
by the National Research Foundation of Korea(NRF) grant funded by the Korea government(MSIT) (No. 2018R1C1B6007009).

\section{Statement of the main results}
\subsection{Working hypotheses} \label{subsec:working-hypotheses}
Let $E$ be a non-CM elliptic curve over $\mathbb{Q}$ of conductor $N$ and $p \geq 5$ a prime such that
\begin{itemize}
\item[(a)] the mod $p$ Galois representation $\overline{\rho} : \mathrm{Gal}(\overline{\mathbb{Q}}/\mathbb{Q}) \to \mathrm{Aut}_{\mathbb{F}_p}(E[p])$ is surjective, and
\item[(b)] the Manin constant is prime to $p$.
\end{itemize}
It is expected that (b) is always true and it actually holds if $E$ has semi-stable reduction at $p$ \cite[Corollary 4.1]{mazur-rational-isogenies}.

Let $K$ be an imaginary quadratic field of discriminant $D_K$ such that $(D_K, Np) = 1$, $D_K$ is odd and $\neq -3$.
Write $$N = N^+ \cdot N^-$$
where a prime divisor of $N^+$ splits in $K/\mathbb{Q}$ and 
a prime divisor of $N^-$ is inert in $K/\mathbb{Q}$.
For a square-free integer $M$, denote by $\nu(M)$ the number of prime divisors of $M$.
We always assume that 
\begin{itemize}
\item[(c)] $N^-$ is square-free,
\item[(d)] if $\nu(N^-)$ is odd, then $E$ has good reduction at $p$, and
\item[(e)] if $\nu(N^-)$ is odd and a prime $q$ satisfies $q \vert N^-$ and $q \equiv \pm 1 \pmod{p}$, then $\overline{\rho}$ is ramified at $q$. (\emph{Condition CR})
\end{itemize}
Hypotheses (a) and (b) also hold for the quadratic twist $E^K $of $E$ by $K$.

Let $T$ be the $p$-adic Tate module of $E$ and $\mathcal{F}_{\mathrm{cl}}$ the classical Selmer structure for $T$.
\subsection{Kurihara numbers} \label{subsec:kurihara-numbers-invariants}
Let $k \geq 1$ be an integer,
$$\mathcal{P}^{\mathrm{cyc}}_k = \left\lbrace q, \textrm{  a prime} : q \nmid Np, q \equiv 1 \Mod{p^k}, a_q(E) \equiv q +1 \Mod{p^k}  \right\rbrace,$$
and write $I^{\mathrm{cyc}}_q = (q - 1, a_q(E) - q -1 ) \mathbb{Z}_p$ for $q \in \mathcal{P}^{\mathrm{cyc}}_k$.
Let $\mathcal{N}^{\mathrm{cyc}}_k$ be the set of square-free products of the primes in $\mathcal{P}^{\mathrm{cyc}}_k$.
For $n \in \mathcal{N}^{\mathrm{cyc}}_k$, write $I^{\mathrm{cyc}}_n = \sum_{q \vert n} I^{\mathrm{cyc}}_q$.

The collection of Kurihara numbers for $E$ is denoted by
$$\kn(E) = \kn = \left\lbrace \widedelta_n \in \mathbb{Z}_p/ I^{\mathrm{cyc}}_n \mathbb{Z}_p \right\rbrace_{n \in \mathcal{N}^{\mathrm{cyc}}_1}$$
and each $\delta_n$ is explicitly built out from modular symbols. See $\S$\ref{subsec:kurihara-numbers} for the precise definition.
The numerical invariants associated to $\kn$ are defined as follows
\begin{align*}
\mathrm{ord} (\kn) & = \mathrm{min} \left\lbrace \nu(n) : n \in \mathcal{N}^{\mathrm{cyc}}_1, \widetilde{\delta}_n \neq 0 \right\rbrace , \\
\mathrm{ord}_p ( \widetilde{\delta}_n ) & = 
\mathrm{max} \left\lbrace j : \widetilde{\delta}_n \in p^j \mathbb{Z}_p/I^{\mathrm{cyc}}_n\mathbb{Z}_p  \right\rbrace , \\
\partial^{(i)}(\widetilde{\boldsymbol{\delta}}) & = 
\mathrm{min} \left\lbrace \mathrm{ord}_p ( \widetilde{\delta}_n ) :  n \in \mathcal{N}^{\mathrm{cyc}}_1 \textrm{ with } \nu(n) = i  \right\rbrace , \\
\partial^{(\infty)}( \kn  ) & = \mathrm{min} \left\lbrace \partial^{(i)}( \kn  )  : i \geq 0 \right\rbrace .
\end{align*}
\subsection{Heegner point Kolyvagin systems}
Suppose that $\nu(N^-)$ is even.
Let $k \geq 1$ be an integer,
$$\mathcal{P}^{\mathrm{ac}}_k = \left\lbrace q, \textrm{  a prime} : q \nmid Np, q \textrm{ inert in } K/\mathbb{Q},  a_q(E) \equiv q +1 \equiv 0 \Mod{p^k}  \right\rbrace,$$
and write $I^{\mathrm{ac}}_q = (q +1, a_q(E)  ) \mathbb{Z}_p$ for $q \in \mathcal{P}^{\mathrm{ac}}_k$.
Let $\mathcal{N}^{\mathrm{ac}}_k$ be the set of square-free products of the primes in $\mathcal{P}^{\mathrm{ac}}_k$.
For $n \in \mathcal{N}^{\mathrm{ac}}_k$, write $I^{\mathrm{ac}}_n = \sum_{q \vert n} I^{\mathrm{ac}}_q$.

The Heegner point Kolyvagin system for $(T, \mathcal{F}_{\mathrm{cl}}, \mathcal{N}^{\mathrm{ac}}_1)$ is the family of cohomology classes
$$\ks^{\mathrm{Heeg}} = \left\lbrace \kappa^{\mathrm{Heeg}}_n \in \mathrm{Sel}_{\mathcal{F}_{\mathrm{cl}}(n)}(K, T/ I^{\mathrm{ac}}_n T) \right\rbrace_{n \in \mathcal{N}^{\mathrm{ac}}_1}$$
satisfying certain relations \cite[Definition 1.2.3]{howard-kolyvagin}
where $\mathcal{F}_{\mathrm{cl}}(n)$ means the classical Selmer structure except transverse local conditions at primes dividing $n$. See $\S$\ref{subsec:heegner-point-kolyvagin-systems} for details.
The numerical invariants associated to $\ks^{\mathrm{Heeg}}$ are defined by
\begin{align*}
\mathrm{ord} (\ks^{\mathrm{Heeg}}) & = \mathrm{min} \left\lbrace \nu(n) : n \in \mathcal{N}^{\mathrm{ac}}_1, \kappa^{\mathrm{Heeg}}_n \neq 0 \right\rbrace , \\
\mathrm{ord}_p( \kappa^{\mathrm{Heeg}}_n ) & = \mathrm{max} \left\lbrace j : \kappa^{\mathrm{Heeg}}_n \in p^j \mathrm{Sel}_{\mathcal{F}_{\mathrm{cl}}(n)}(K, T / I^{\mathrm{ac}}_n T)  \right\rbrace , \\
\partial^{(i)}( \ks^{\mathrm{Heeg}} ) & = \mathrm{min} \left\lbrace \mathrm{ord}_p( \kappa^{\mathrm{Heeg}}_n ) :  n \in \mathcal{N}^{\mathrm{ac}}_1 \textrm{ with } \nu(n) = i  \right\rbrace , \\
\partial^{(\infty)}( \ks^{\mathrm{Heeg}}  ) & = \mathrm{min} \left\lbrace \partial^{(i)}( \ks^{\mathrm{Heeg}}  )  : i \geq 0 \right\rbrace .
\end{align*}
\subsection{Bipartite Euler systems}
Suppose that $\nu(N^-)$ is odd.
Let $k \geq 1$ be an integer,
\begin{align*}
\mathcal{P}^{\mathrm{adm}}_k = \left\lbrace q, \textrm{  a prime} : \right. & q \nmid Np, q \textrm{ inert in } K/\mathbb{Q},  q\not\equiv \pm 1 \Mod{p},  \\
& \left. a_q(E) \equiv \epsilon_q \cdot (q +1) \Mod{p^k} \textrm{ with } \epsilon_q = 1 \textrm{ or } -1  \right\rbrace ,
\end{align*}
and $I^{\mathrm{adm}}_\ell = ( a_q(E) - \epsilon_q \cdot (q +1)  ) \mathbb{Z}_p$.
Let $\mathcal{N}^{\mathrm{adm}}_k$ be the set of square-free products of the primes in $\mathcal{P}^{\mathrm{adm}}_k$.
For $n \in \mathcal{N}^{\mathrm{adm}}_k$, write $I^{\mathrm{adm}}_n = \sum_{\ell \vert n} I^{\mathrm{adm}}_\ell$.

Let $\mathcal{N}^{\mathrm{def}}_k \subseteq \mathcal{N}^{\mathrm{adm}}_k$ 
be the subset defined by putting the additional condition: every $n \in \mathcal{N}^{\mathrm{def}}_k$ satisfies that $\nu(nN^-)$ is odd.
Then $\mathcal{N}^{\mathrm{ind}}_k$ is defined by $\mathcal{N}^{\mathrm{adm}}_k \setminus \mathcal{N}^{\mathrm{def}}_k$ so that
every $n \in \mathcal{N}^{\mathrm{ind}}_k$ satisfies that $\nu(nN^-)$ is even.
A bipartite Euler system for $(T, \mathcal{F}_{\mathrm{cl}}, \mathcal{N}^{\mathrm{adm}}_1)$  consists of two families
\[
\xymatrix{
\lambdab^{\mathrm{bip}} =  \left\lbrace \lambda^{\mathrm{bip}}_n \in \mathbb{Z}_p/I^{\mathrm{adm}}_n \mathbb{Z}_p \right\rbrace_{n \in \mathcal{N}^{\mathrm{def}}_1}, & \ks^{\mathrm{bip}} = \left\lbrace \kappa^{\mathrm{bip}}_n \in \mathrm{Sel}_{\mathcal{F}_{\mathrm{cl}}(n)}(K, T/I^{\mathrm{adm}}_n T) \right\rbrace_{n \in \mathcal{N}^{\mathrm{ind}}_1}
}
\]
with the connection via the first and second explicit reciprocity laws (Definition \ref{defn:bipartite-euler-systems}) where 
$\lambdab^{\mathrm{bip}}$ is constructed from the values of congruent families of quaternionic automorphic forms varying definite quaternion algebras and
$\ks^{\mathrm{bip}}$ is constructed from Heegner points on \emph{different} Shimura curves varying indefinite quaternion algebras.
For the bipartite setting, $\mathcal{F}_{\mathrm{cl}}(n)$ means the classical Selmer structure except ordinary local conditions at primes dividing $n$ as reviewed in $\S$\ref{subsubsec:selmer-structure}. See $\S$\ref{sec:bipartite-euler-systems} for further details.
The numerical invariants associated to $\lambdab^{\mathrm{bip}}$ are defined by
\begin{align*}
\mathrm{ord} (\lambdab^{\mathrm{bip}}) & = \mathrm{min} \left\lbrace \nu(n) : n \in \mathcal{N}^{\mathrm{def}}_1, \lambda^{\mathrm{bip}}_n \neq 0 \right\rbrace , \\
\mathrm{ord}_p(  \lambda^{\mathrm{bip}}_n ) & = \mathrm{max} \left\lbrace j : \lambda^{\mathrm{bip}}_n \in p^j \mathbb{Z}_p/I^{\mathrm{adm}}_n \mathbb{Z}_p  \right\rbrace , \\
\partial^{(i)}( \lambdab^{\mathrm{bip}} ) & = 
\mathrm{min} \left\lbrace \mathrm{ord}_p(  \lambda^{\mathrm{bip}}_n ) : n \in \mathcal{N}^{\mathrm{def}}_1 \textrm{ with (even) } \nu(n) = i  \right\rbrace , \\
\partial^{(\infty)}( \lambdab^{\mathrm{bip}} ) & = \mathrm{min} \left\lbrace \partial^{(i)}( \lambdab^{\mathrm{bip}} )  : i \geq 0 \right\rbrace .
\end{align*}
Our formulation of bipartite Euler systems works even when an elliptic curve has supersingular reduction at $p$ \emph{and} $p$ is inert in $K/\mathbb{Q}$.
\subsection{The non-triviality questions} \label{subsec:non-triviality}
The non-vanishing of $\kn$, the non-triviality of $\ks^{\mathrm{Heeg}}$, and the non-vanishing of $\lambdab^{\mathrm{bip}}$ play  important roles for structure theorems.
In \cite{kim-structure-selmer}, we proved the equivalence between the the Iwasawa main conjecture localized at the augmentation ideal and the non-vanishing of $\kn$. 
Extending this idea to the cases of $\ks^{\mathrm{Heeg}}$ and $\lambdab^{\mathrm{bip}}$, we obtain the following result.
\begin{thm} \label{thm:non-triviality-total}
We keep all the working hypotheses in $\S$\ref{subsec:working-hypotheses}.
\begin{enumerate}
\item[(IMC)] The Iwasawa main conjecture localized at the augmentation ideal holds if and only if $\kn$ does not vanish.
In this case, the structure of $\mathrm{Sel}(\mathbb{Q}, E[p^\infty])$ is determined by $\kn$.
\item[(HPMC)] Suppose that $\nu(N^-)$ is even.
If $E$ has good ordinary reduction at $p$, then the Heegner point main conjecture localized at the augmentation ideal holds if and only if $\ks^{\mathrm{Heeg}}$ is non-trivial.
In this case, if the ranks of $E(\mathbb{Q})$ and $E^{K}(\mathbb{Q})$ only differ by one, then
the structure of $\mathrm{Sel}(K, E[p^\infty])$ is determined by $\ks^{\mathrm{Heeg}}$.
\item[(AMC)]  Suppose that $\nu(N^-)$ is odd.
If $E$ has good ordinary reduction at $p$, then the anticyclotomic main conjecture localized at the augmentation ideal holds if and only if $\lambdab^{\mathrm{bip}}$ does not vanish.
In this case, if the mod $p^k$ reduction of $\lambdab^{\mathrm{bip}}$ is free for every $k \geq 1$ in the sense of Definition 
\ref{defn:freeness-bipartite}, then the structure of $\mathrm{Sel}(K, E[p^\infty])$ is determined by $\lambdab^{\mathrm{bip}}$.
\end{enumerate}
\end{thm}
\begin{proof}
For the equivalence statements, see Corollary \ref{cor:non-triviality-kn} for (IMC), Corollary \ref{cor:non-triviality-heegner} for (HPMC), and Corollary \ref{cor:non-triviality-bipartite} for (AMC), respectively.

For the structure theorems, see Theorem \ref{thm:structure-kn} for $\kn$, Theorem \ref{thm:structure-kolyvagin} for $\ks^{\mathrm{Heeg}}$, and Theorem \ref{thm:structure-bipartite} for $\lambdab^{\mathrm{bip}}$, respectively.
\end{proof}
In order to avoid the confusion of the reader, we emphasize that Theorem \ref{thm:structure-bipartite}, Corollary \ref{cor:non-triviality-heegner}, and Corollary \ref{cor:non-triviality-bipartite} are the main contribution of this article. In other words, we prove
\[
\xymatrix{
 {\substack{\textrm{Perrin-Riou's Heegner point main conjecture} \\\textrm{at the augmentation ideal} \\\textrm{\cite{perrin-riou-heegner}} } } 
  \ar@{<=>}[rr]^-{\textrm{Cor. \ref{cor:non-triviality-heegner}}} &  &  {\substack{\textrm{the non-triviality of} \\\textrm{the Heegner point Kolyvagin system} \\\textrm{(Kolyvagin's conjecture \cite{kolyvagin-selmer}),} } }   
}
\]
and its definite analogue
\[
\xymatrix{
 {\substack{\textrm{Bertolini--Darmon's anticyclotomic main conjecture} \\\textrm{at the augmentation ideal} \\\textrm{\cite{bertolini-darmon-imc-2005}} } } 
  \ar@{<=>}[rr]^-{\textrm{Cor. \ref{cor:non-triviality-bipartite}}}   & & {\substack{\textrm{the non-triviality of} \\\textrm{the bipartite Euler system} \\\textrm{\cite{howard-bipartite}.} } }   
}
\]
We also prove the definite analogue of Kolyvagin's structure theorem
\[
\xymatrix{
  {\substack{\textrm{the non-triviality of} \\\textrm{the (free) bipartite Euler system} \\\textrm{\cite{howard-bipartite}} } }
  \ar@{=>}[rr]^-{\textrm{Thm. \ref{thm:structure-bipartite}}}_-{\textrm{Thm. \ref{thm:structure-bipartite-non-free}}}  & &  {\substack{\textrm{the structure of Selmer groups} \\\textrm{in the definite setting} \\\textrm{(cf. \cite{kolyvagin-selmer}).} } }   
}
\]
It is remarkable that the first implication (Corollary \ref{cor:non-triviality-heegner}) refines the standard argument to obtain the strong version of the rank one $p$-converse to the theorem of Gross--Zagier and Kolyvagin from the Heegner point main conjecture \cite{burungale-tian-p-converse, wan-heegner}.
Theorem \ref{thm:structure-kn}  and Corollary \ref{cor:non-triviality-kn}  are proved in \cite{kim-structure-selmer}, and Theorem \ref{thm:structure-kolyvagin}  is proved in \cite{kolyvagin-selmer}.

Theorem \ref{thm:non-triviality-total} provides us with numerous examples satisfying the non-triviality statements which are used in Theorems \ref{thm:main-higher-gross-zagier-easy} and \ref{thm:main-higher-waldspurger-easy}.
We do not take the most general version here in order to make the statement simpler.
\begin{cor}
Let $E$ be a semi-stable elliptic curve over $\mathbb{Q}$ with good ordinary reduction at $p \geq 5$ such that $\overline{\rho}$ is surjective.
\begin{enumerate}
\item[($\kn$)] Then $\kn$ does not vanish.
\item[($\ks^{\mathrm{Heeg}}$)] Let $K$ be an imaginary quadratic field in $\S$\ref{subsec:working-hypotheses} with even $\nu(N^-)$.
If $\overline{\rho}$ is ramified at all primes dividing $N^-$, then $\ks^{\mathrm{Heeg}}$ is non-trivial.
\item[($\lambdab^{\mathrm{bip}}$)] Let $K$ be an imaginary quadratic field in $\S$\ref{subsec:working-hypotheses} with odd $\nu(N^-)$.
Assume that
\begin{itemize}
\item $a_p(E) \not\equiv 1 \pmod{p}$ if $p$ splits in $K$, and
\item $a_p(E) \not\equiv \pm 1 \pmod{p}$ if $p$ is inert in $K$.
\end{itemize}
If $\overline{\rho}$ is ramified at all primes dividing $N^-$, then $\lambdab^{\mathrm{bip}}$ does not vanish.
\end{enumerate}
\end{cor}
\begin{proof}
See  \cite{sakamoto-p-selmer, kim-structure-selmer} for ($\kn$),   \cite{wei-zhang-mazur-tate} for ($\ks^{\mathrm{Heeg}}$), and \cite{burungale-castella-kim} for ($\lambdab^{\mathrm{bip}}$).
\end{proof}
Recently, Burungale--Castella--Grossi--Skinner independently obtained the non-triviality of $\ks^{\mathrm{Heeg}}$ and Kato's Kolyvagin systems without the large image assumption \cite{burungale-castella-grossi-skinner-indivisibility}. In their argument, the full Heegner point main conjecture and the full Iwasawa main conjecture are used, respectively. Also, the structure theorems are not studied therein yet.

\subsection{Higher Gross--Zagier and Waldspurger formulas}
Denote by $M_{/\mathrm{div}}$ the quotient of $M$ by its maximal divisible submodule.

Our ``higher" Gross--Zagier formula relates $\ks^{\mathrm{Heeg}}$ with $\kn(E)$ and $\kn(E^K)$ and reveals the exact size of the Selmer group as follows.
\begin{thm} \label{thm:main-higher-gross-zagier-easy}
We keep all the working hypotheses in $\S$\ref{subsec:working-hypotheses} with even $\nu(N^-)$.
If
 $\ks^{\mathrm{Heeg}}$ is non-trivial, and
 $\kn(E)$ and $\kn(E^K)$ are both non-zero, 
then 
\begin{align*}
\mathrm{ord}(\ks^{\mathrm{Heeg}})+ 1 
 & = \mathrm{max} \left\lbrace \mathrm{ord} (\kn(E)) ,\mathrm{ord} (\kn(E^K)) \right\rbrace \\
& =  \mathrm{max} \left\lbrace \mathrm{cork}_{\mathbb{Z}_p} \mathrm{Sel}(\mathbb{Q}, E[p^\infty]) , \mathrm{cork}_{\mathbb{Z}_p} \mathrm{Sel}(\mathbb{Q}, E^K[p^\infty]) \right\rbrace .
\end{align*}
If we further assume that
$\mathrm{cork}_{\mathbb{Z}_p} \mathrm{Sel}(\mathbb{Q}, E[p^\infty])$ and $\mathrm{cork}_{\mathbb{Z}_p} \mathrm{Sel}(\mathbb{Q}, E^K[p^\infty])$ differ by 1, then
\begin{align*}
\mathrm{ord}(\ks^{\mathrm{Heeg}})
 & = \mathrm{min} \left\lbrace \mathrm{ord} (\kn(E)) ,\mathrm{ord} (\kn(E^K)) \right\rbrace \\
& =  \mathrm{min} \left\lbrace \mathrm{cork}_{\mathbb{Z}_p} \mathrm{Sel}(\mathbb{Q}, E[p^\infty]) , \mathrm{cork}_{\mathbb{Z}_p} \mathrm{Sel}(\mathbb{Q}, E^K[p^\infty]) \right\rbrace ,
\end{align*}
and
\begin{align} \label{eqn:main-higher-gross-zagier-easy}
\begin{split}
2 \cdot \left( \partial^{( \mathrm{ord}(\ks^{\mathrm{Heeg}}) )} ( \ks^{\mathrm{Heeg}} )
- \partial^{( \infty )} ( \ks^{\mathrm{Heeg}} ) \right)
 = \ &
\partial^{( \mathrm{ord} (\kn(E)) )} ( \kn(E) )
- \partial^{( \infty )} ( \kn(E) ) \\
& +
\partial^{( \mathrm{ord} (\kn(E^K)) )} ( \kn(E^K) )
- \partial^{( \infty )} ( \kn(E^K) ) \\
= \ & \mathrm{length}_{\mathbb{Z}_p}  \left( \mathrm{Sel}(K, E[p^\infty])_{/\mathrm{div}} \right).
\end{split}
\end{align}
\end{thm}
\begin{proof}
This is a piece of the comparison between Theorem  \ref{thm:structure-kn} and Theorem \ref{thm:structure-kolyvagin}, which is given in $\S$\ref{subsec:structural-gross-zagier}.
\end{proof}
\begin{rem} \label{rem:main-higher-gross-zagier-easy}
\begin{enumerate}
\item
The additional assumption on Selmer coranks comes from Kolyvagin's structure theorem \cite{kolyvagin-selmer}, \cite[Remark 18]{wei-zhang-mazur-tate}.
See \cite{rubin-silverberg-twists-rank-four} for the variation of ranks in quadratic twist families of elliptic curves.
\item 
Conceptually, the RHS of (\ref{eqn:main-higher-gross-zagier-easy}) corresponds to the derivative of the Rankin--Selberg product $L$-function, and the LHS of (\ref{eqn:main-higher-gross-zagier-easy}) corresponds to the height of Heegner points. See $\S$\ref{subsec:classical-gross-zagier} for the comparison with classical Gross--Zagier formula. \item
Let $\ks^{\mathrm{Kato}}(E)$ be Kato's Kolyvagin system for $E$.
Since $\kn(E)$ is the image of $\ks^{\mathrm{Kato}}(E)$ under the dual exponential map, Theorem \ref{thm:main-higher-gross-zagier-easy} also explicitly relates  $\ks^{\mathrm{Heeg}}$ to $\ks^{\mathrm{Kato}}(E)$ and $\ks^{\mathrm{Kato}}(E^K)$ from the viewpoint of Perrin-Riou's conjecture \cite{perrin-riou-rational-pts, bertolini-darmon-venerucci, burungale-skinner-tian-wan, kazim-pollack-sasaki}.
\end{enumerate}
\end{rem}
Our ``higher" Waldspurger formula relates $\lambdab^{\mathrm{bip}}$ with $\kn(E)$ and $\kn(E^K)$ and reveals the exact size of the Selmer group again.
\begin{thm} \label{thm:main-higher-waldspurger-easy}
We keep all the working hypotheses in $\S$\ref{subsec:working-hypotheses} with odd $\nu(N^-)$.
Assume that the mod $p^k$ reduction of $\lambdab^{\mathrm{bip}}$ is free for every $k \geq 1$ in the sense of Definition \ref{defn:freeness-bipartite}. 
If
 $\lambdab^{\mathrm{bip}}$ is non-zero, and
 $\kn(E)$ and $\kn(E^K)$ are non-zero,
then 
$$\mathrm{ord} ( \lambdab^{\mathrm{bip}} ) =  \mathrm{ord} (\kn(E)) + \mathrm{ord} (\kn(E^K)) = \mathrm{cork}_{\mathbb{Z}_p} \mathrm{Sel}(K, E[p^\infty])$$
and
\begin{align} \label{eqn:main-higher-waldspurger-easy}
\begin{split}
2 \cdot \left( \partial^{( \mathrm{ord}( \lambdab^{\mathrm{bip}} ) )} ( \lambdab^{\mathrm{bip}} )
- \partial^{( \infty )} ( \lambdab^{\mathrm{bip}} ) \right)
 = \ &
\partial^{( \mathrm{ord} (\kn(E)) )} ( \kn(E) )
- \partial^{( \infty )} ( \kn(E) ) \\
& +
\partial^{( \mathrm{ord} (\kn(E^K)) )} ( \kn(E^K) )
- \partial^{( \infty )} ( \kn(E^K) ) \\
= \ & \mathrm{length}_{\mathbb{Z}_p} \left( \mathrm{Sel}(K, E[p^\infty])_{/\mathrm{div}} \right). 
\end{split}
\end{align}
\end{thm}
\begin{proof}
It follows from the comparison between the structure theorem of Selmer groups via bipartite Euler systems (Theorem \ref{thm:structure-bipartite}) and Theorem \ref{thm:structure-kn}, which is given in $\S$\ref{subsec:structural-waldspurger}.
\end{proof}
\begin{rem}
Even without the freeness assumption, we are still able to prove the mod $p^k$ version of Theorem \ref{thm:main-higher-waldspurger-easy} for every $k \geq 1$.
See (\ref{eqn:main-higher-waldspurger-easy-mod-p-k}).
\end{rem}

\subsection{The connection with the classical Gross--Zagier formula} \label{subsec:classical-gross-zagier}
Let $K$ be an imaginary quadratic field of discriminant $D_K$ as in $\S$\ref{subsec:working-hypotheses} with $N^- = 1$.
Assume that $E$ has rank one and its quadratic twist $E^K$ by $K$ has rank zero.
Thus, for a good reduction prime $p >3$ with $E(K)[p] = 0$,  the $p$-part of the Gross--Zagier formula \cite{gross-zagier-original} is equivalent to
\begin{equation} \label{eqn:gross-zagier-orinigal-modified-2}
 \mathrm{ord}_p \left(   \dfrac{L'(E,1)}{\Omega^+_E \cdot \mathrm{Reg}(E/\mathbb{Q})   }   \right) + \mathrm{ord}_p \left(  \dfrac{L(E^K,1)}{ \Omega^+_{E^K}   } \right) = 2 \cdot \mathrm{ord}_p \left( [E(K): \mathbb{Z} \cdot P_K] \right) 
\end{equation}
where $P_K$ is the Heegner point over $K$. 
\begin{rem}
If $N^- \neq 1$ and $\nu(N^-)$ is just even, then the Heegner point on a Shimura curve should be used in (\ref{eqn:gross-zagier-orinigal-modified-2}), and the Tamagawa factors at primes dividing $N^-$ appears in (\ref{eqn:gross-zagier-orinigal-modified-2}) as in \cite[(7.4.c)]{jetchev-skinner-wan}. 
\end{rem}
For simplicity, we assume that both $E(\mathbb{Q}_p)[p^\infty]$ and $E^K(\mathbb{Q}_p)[p^\infty]$ are trivial.
Then one of our formulas  (\ref{eqn:main-higher-gross-zagier-easy}) can be regarded as a conceptual generalization of (\ref{eqn:gross-zagier-orinigal-modified-2}) with equalities
\begin{equation} \label{eqn:modular-symbols-regulators}
\xymatrix{
\partial^{(1)} ( \kn(E) ) = \mathrm{ord}_p \left( \dfrac{L'(E,1)}{\Omega^+_{E} \cdot \mathrm{Reg}(E/\mathbb{Q}) } \right) , & \partial^{(0)} ( \kn(E^K) ) = \mathrm{ord}_p \left( \dfrac{L(E^K,1)}{\Omega^+_{E^K}} \right)
}
\end{equation}
``up to Tamagawa factors" since we also expect that
\begin{align*}
\partial^{( \infty )} ( \ks^{\mathrm{Heeg}} ) & = \partial^{( \infty )} ( \kn(E) ) + \partial^{( \infty )} ( \kn(E^K) ) \\
& = \textrm{the sum of the $p$-valuations of local Tamagawa factors of $E$ over $K$} .
\end{align*}
See \cite{jetchev-global-divisibility, kazim-tamagawa, wei-zhang-cdm, kim-structure-selmer, burungale-castella-grossi-skinner-indivisibility} for the details of the Tamagawa factor issue.
The first equality of (\ref{eqn:modular-symbols-regulators})  explicitly shows how the derivative of the $L$-function and the regulator are absorbed together in modular symbols.

\section{Quick review of Kurihara numbers and Heegner point Kolyvagin systems}
\subsection{Modular symbols and Kurihara numbers} \label{subsec:kurihara-numbers}
We review the structural refinement of Birch and Swinnerton-Dyer conjecture in \cite{kim-structure-selmer}.
\subsubsection{Modular symbols and Kurihara numbers}
Let $p \geq 5$ be a prime and $E$ an elliptic curve over $\mathbb{Q}$ satisfying (a) and (b) of the working hypotheses in $\S$\ref{subsec:working-hypotheses}.
Let $f = \sum_{n \geq 1} a_n(E) q^n \in S_2(\Gamma_0(N))$ be the newform corresponding to $E$ \cite{bcdt}.
For each $\dfrac{a}{b} \in \mathbb{Q}$, the modular symbol $\left[\dfrac{a}{b}\right]^+$ is defined by equality
$$2 \pi \cdot \int^{\infty}_0 f(\dfrac{a}{b}+ iy )dy = \left[\dfrac{a}{b}\right]^+ \cdot \Omega^+_E + \left[\dfrac{a}{b}\right]^- \cdot \sqrt{-1} \cdot \Omega^-_E$$
where $\left[\dfrac{a}{b}\right]^+$ and $\left[\dfrac{a}{b}\right]^-$ are rational numbers where $\Omega^{\pm}_E$ is the real and imaginary N\'{e}ron period of $E$, respectively \cite[(1.1)]{mazur-tate}.
Under our working hypotheses, $\left[\dfrac{a}{b}\right]^+ \in \mathbb{Z}_{(p)}$.

For each prime $\ell \in \mathcal{P}^{\mathrm{cyc}}_k$, we fix a primitive root $\eta_\ell$ mod $\ell$ and define
 $\mathrm{log}_{\eta_\ell}(a) \in \mathbb{Z}/(\ell-1)$ by $\eta^{ \mathrm{log}_{\eta_\ell}(a)}_\ell \equiv a \pmod{\ell}$.
For each $n \in \mathcal{N}^{\mathrm{cyc}}_1$, the \textbf{(mod $I^{\mathrm{cyc}}_n$) Kurihara number at $n$} is defined by
$$\widetilde{\delta}_n  = \sum_{a \in (\mathbb{Z}/n\mathbb{Z})^\times } \left( \overline{\left[ \dfrac{a}{n} \right]^+}  \cdot   \prod_{\ell \mid n} \overline{\mathrm{log}_{\eta_\ell} (a)} \right) \in \mathbb{Z}_p/I^{\mathrm{cyc}}_n\mathbb{Z}_p$$
where $\overline{\left[ \dfrac{a}{n} \right]^+}$ is the mod $I^{\mathrm{cyc}}_n$ reduction of $\left[ \dfrac{a}{n} \right]^+$ and 
$\overline{\mathrm{log}_{\eta_\ell} (a)}$ is also the mod $I^{\mathrm{cyc}}_n$ reduction of $\mathrm{log}_{\eta_\ell} (a)$.
Note that $\widetilde{\delta}_n$ is well-defined up to $(\mathbb{Z}_p/I^{\mathrm{cyc}}_n\mathbb{Z}_p)^\times$.
When $n=1$, we have
$$\widetilde{\delta}_1 = [0]^+ = \dfrac{L(E, 1)}{\Omega^+_E} \in \mathbb{Z}_{(p)} .$$
The \textbf{collection of Kurihara numbers} is defined by 
$$\widetilde{\boldsymbol{\delta}} = \left\lbrace \widetilde{\delta}_n \in \mathbb{Z}_p/I^{\mathrm{cyc}}_n \mathbb{Z}_p : n \in \mathcal{N}^{\mathrm{cyc}}_1 \right\rbrace .$$
We assume that $\mathrm{ord} (\widetilde{\boldsymbol{\delta}}) < \infty$, i.e. the collection of Kurihara numbers does not vanish identically.
\subsubsection{The modular symbol description of Selmer groups}
Reminding the numerical invariants in $\S$\ref{subsec:kurihara-numbers-invariants},
the following structure theorem is proved in \cite{kim-structure-selmer}.
\begin{thm} \label{thm:structure-kn}
Let $E$ be an elliptic curve over $\mathbb{Q}$ and $p \geq 5$ a prime  satisfying (a) and (b) of the working hypotheses in $\S$\ref{subsec:working-hypotheses}.
If $\mathrm{ord} (\kn) < \infty$,
then
\begin{enumerate}
\item $\mathrm{ord}(\kn) =  \mathrm{cork}_{\mathbb{Z}_p}\mathrm{Sel}(\mathbb{Q}, E[p^\infty])$,
\item $\mathrm{Fitt}_{i, \mathbb{Z}_p} \left(  \mathrm{Sel}(\mathbb{Q}, E[p^\infty])^\vee \right) = p^{\partial^{(i)} (\widetilde{\boldsymbol{\delta}} ) - \partial^{(\infty)} (\widetilde{\boldsymbol{\delta}} )} \mathbb{Z}_p$
for all $i \geq \mathrm{ord}(\kn)$ with $i \equiv \mathrm{ord}(\kn) \pmod{2}$, and
\item $\partial^{(\mathrm{ord}(\kn))} (\widetilde{\boldsymbol{\delta}} ) - \partial^{(\infty)} (\widetilde{\boldsymbol{\delta}} ) = \mathrm{length}_{\mathbb{Z}_p} \left( \mathrm{Sel}(\mathbb{Q}, E[p^\infty])_{/\mathrm{div}} \right)$.
\end{enumerate}
In other words, we have
$$\mathrm{Sel}(\mathbb{Q}, E[p^\infty])  \simeq \left( \mathbb{Q}_p/\mathbb{Z}_p \right)^{\oplus \mathrm{ord}(\kn)} \oplus \bigoplus_{i\geq 1} \left( \mathbb{Z}/p^{(  \partial^{(\mathrm{ord}(\kn) + 2(i-1))} (\widetilde{\boldsymbol{\delta}} ) - \partial^{(\mathrm{ord}(\kn) + 2i)} (\widetilde{\boldsymbol{\delta}} ) )/2}\mathbb{Z} \right)^{\oplus 2}  .$$
\end{thm}
\begin{proof}
See \cite{kim-structure-selmer}.
\end{proof}

\subsection{Heegner point Kolyvagin systems} \label{subsec:heegner-point-kolyvagin-systems}
We review the main result of \cite{kolyvagin-selmer}.
We keep all the working hypotheses in $\S$\ref{subsec:working-hypotheses} and assume that $\nu(N^-)$ is even.
\subsubsection{Convention}
Let $X^{N^-}(N^+)$ be the (compactified) Shimura curve associated to the indefinite quaternion algebra of discriminant $N^-$ and an Eichler order of level $N^+$.
Let $\ks^{\mathrm{Heeg}}$ be the \textbf{Heegner point Kolyvagin system} coming from Heegner points on $X^{N^-}(N^+)$. 
See \cite{howard-kolyvagin, nekovar-euler-systems, burungale-castella-kim} for the explicit construction of  $\ks^{\mathrm{Heeg}}$.
Let $\mathrm{Sel}(K, E[p^\infty])^{\pm}$ be the eigenspace with eigenvalue $\pm 1$ of $\mathrm{Sel}(K, E[p^\infty])$ with respect to the complex conjugation, respectively.
Since $p >2$, we have decomposition
$$\mathrm{Sel}(K, E[p^\infty]) = \mathrm{Sel}(K, E[p^\infty])^{+} \oplus \mathrm{Sel}(K, E[p^\infty])^{-} .$$
Write
\[
\xymatrix{
r(E/K) = \mathrm{cork}_{\mathbb{Z}_p} \mathrm{Sel}(K, E[p^\infty]), & r^{\pm}(E/K) = \mathrm{cork}_{\mathbb{Z}_p} \mathrm{Sel}(K, E[p^\infty])^{\pm} .
}
\]
Denote by  $W(E/\mathbb{Q})$ the root number of $E$ over $\mathbb{Q}$.
\subsubsection{Structure theorem}
\begin{thm} \label{thm:kolyvagin-vanishing-order}
Suppose that $\ks^{\mathrm{Heeg}}$ is non-trivial.
Then the following statements are valid.
\begin{enumerate}
\item $\mathrm{ord}(\ks^{\mathrm{Heeg}}) = \mathrm{max} \lbrace r^+(E/K), r^-(E/K) \rbrace - 1$.
\item $r^{W(E/\mathbb{Q}) \cdot (-1)^{\mathrm{ord}(\ks^{\mathrm{Heeg}}) +1}} (E/K)= \mathrm{ord}(\ks^{\mathrm{Heeg}}) +1$.
\item $\mathrm{ord}(\ks^{\mathrm{Heeg}}) - r^{W(E/\mathbb{Q}) \cdot (-1)^{\mathrm{ord}(\ks^{\mathrm{Heeg}}) }} (E/K) \geq 0$ and is even.
\end{enumerate}
\end{thm}
\begin{proof}
See \cite[Theorem 4]{kolyvagin-selmer} and also \cite[Theorems 1.2 and 11.2]{wei-zhang-mazur-tate}.
\end{proof}
By Theorem \ref{thm:kolyvagin-vanishing-order}, we have  
\begin{equation} \label{eqn:corank-difference}
\mathrm{ord}(\ks^{\mathrm{Heeg}}) - r^{ W(E/\mathbb{Q}) \cdot (-1)^{\mathrm{ord}(\ks^{\mathrm{Heeg}}) }} (E/K)  = \vert r^{+}(E/K) - r^{-}(E/K) \vert - 1 .
\end{equation}
In particular, 
 $\mathrm{ord}(\ks^{\mathrm{Heeg}}) - r^{ W(E/\mathbb{Q}) \cdot (-1)^{\mathrm{ord}(\ks^{\mathrm{Heeg}}) }} (E/K) = 0$ if and only if
$\vert r^{+}(E/K) - r^{-}(E/K) \vert =  1$.
Write
\begin{equation} \label{eqn:structure-selmer-before}
\mathrm{Sel}(K, E[p^\infty])^{\pm}_{/\mathrm{div}} \simeq \bigoplus_{i \geq 1} \left( \mathbb{Z} / p^{a^{\pm}_i} \mathbb{Z} \right)^{\oplus 2}
\end{equation}
where $a^{\pm}_1 \geq  a^{\pm}_2 \geq  \cdots$.
\begin{thm} \label{thm:structure-kolyvagin}
Suppose that $\ks^{\mathrm{Heeg}}$ is non-trivial.
Then $a^{\pm}_i$'s in (\ref{eqn:structure-selmer-before}) are determined by $\ks^{\mathrm{Heeg}}$ as follows:
\begin{align*}
a^{W(E/\mathbb{Q}) \cdot (-1)^{\mathrm{ord}(\ks^{\mathrm{Heeg}})+1}}_i & = \partial^{(\mathrm{ord}(\ks^{\mathrm{Heeg}}) + 2i -1)} (\ks^{\mathrm{Heeg}}) - \partial^{(\mathrm{ord}(\ks^{\mathrm{Heeg}}) + 2i)} (\ks^{\mathrm{Heeg}} ) , \\
a^{ W(E/\mathbb{Q}) \cdot (-1)^{\mathrm{ord}(\ks^{\mathrm{Heeg}})}}_{i + \vert r^{+}(E/K) - r^{-}(E/K) \vert - 1 } & = \partial^{(\mathrm{ord}(\ks^{\mathrm{Heeg}}) + 2i -2)} (\ks^{\mathrm{Heeg}}) - \partial^{(\mathrm{ord}(\ks^{\mathrm{Heeg}}) + 2i-1)} (\ks^{\mathrm{Heeg}} )
\end{align*}
for $i \geq 1$.
\end{thm}
\begin{proof}
See \cite[Theorem 1]{kolyvagin-selmer}. See also \cite[Remark 18]{wei-zhang-mazur-tate}.
\end{proof}
\begin{rem} \label{rem:error-term-heegner}
We do not know whether $a^{  W(E/\mathbb{Q}) \cdot (-1)^{\mathrm{ord}(\ks^{\mathrm{Heeg}})}}_i$'s are determined by $\ks^{\mathrm{Heeg}}$ or not
for $1 \leq i \leq  \vert r^{+}(E/K) - r^{-}(E/K) \vert $.
Write
$$\mathrm{Err} = \bigoplus_i \mathbb{Z}/p^{a^{ W(E/\mathbb{Q}) \cdot (-1)^{\mathrm{ord}(\ks^{\mathrm{Heeg}})}}_i}\mathbb{Z}$$
where $i$ runs over $1 \leq i \leq  \vert r^{+}(E/K) - r^{-}(E/K) \vert $. 
Then we have
$$\mathrm{length}_{\mathbb{Z}_p} \left( \mathrm{Sel}(K, E[p^\infty])_{/\mathrm{div}} \right) = 2 \cdot \left( \partial^{(\mathrm{ord}(\ks^{\mathrm{Heeg}}))}(  \ks^{\mathrm{Heeg}} )  - \partial^{(\infty)}(  \ks^{\mathrm{Heeg}} ) + \mathrm{length}_{\mathbb{Z}_p}( \mathrm{Err} )  \right) .$$
Note that $\vert r^{+}(E/K) - r^{-}(E/K) \vert$ can be written purely in terms of $\ks^{\mathrm{Heeg}}$ by (\ref{eqn:corank-difference}).
\end{rem}

\section{Bipartite Euler systems for elliptic curves and the structure theorem} \label{sec:bipartite-euler-systems}
In this section, we deal with the counterpart of Kolyvagin's work on the structure of Selmer groups reviewed in \S\ref{subsec:heegner-point-kolyvagin-systems}.
Let $E$ be an elliptic curve over the rationals, $p \geq 5$ a \emph{good} reduction prime for $E$, and $K$ an imaginary quadratic field.
When the (generalized) Heegner hypothesis for $(E,K)$ is satisfied, 
Theorem \ref{thm:structure-kolyvagin} describes the module structure of 
$\mathrm{Sel}(K, E[p^\infty])$
in terms of the divisibilities of the (non-trivial) Heegner point Kolyvagin system under mild assumptions as we have seen.

The opposite case is studied in this section. In other words, the root number of $E/K$ is 1.
The main result is the description of the structure of $\mathrm{Sel}(K, E[p^\infty])$ in terms of the values of certain families of quaternionic automorphic forms at Gross points (Theorem \ref{thm:structure-bipartite}).

\subsection{Formalism of the theory of bipartite Euler systems} \label{subsec:formalism-bipartite}
We first review the theory of bipartite Euler systems \emph{over the imaginary quadratic field} closely following \cite{howard-bipartite}.
Then we prove the structure theorem of Selmer groups (Theorem \ref{thm:structure-pk-bipartite}).
We fix one principal Artinian local ring $R$ as the coefficient ring in this subsection.
\subsubsection{Preliminaries}
Let $K$ be an imaginary quadratic field and $G_K$ the absolute Galois group of $K$.
Let $R$ be a principal Artinian local ring with maximal ideal $\mathfrak{m}$ and residue characteristic $p > 3$ (e.g. $R = \mathbb{Z}/p^k\mathbb{Z} $).
Let $T$ be a free $R$-module of rank two endowed with continuous action of $G_K$ such that there exists a perfect, $G_K$-equivariant, and alternating pairing 
$$T  \times T \to R(1).$$
\begin{rem} \label{rem:notation-T}
The notation $T$ in this subsection is \emph{not} a Galois representation with a coefficient ring of characteristic zero.
\end{rem}
If $B$ is any $R$-module and $b \in B$ we define $\mathrm{ind}(b, B)$, the index of divisibility of $b$ in $B$,
to be the largest $k \leq \infty$ such that $b \in \mathfrak{m}^k B$.
\subsubsection{Selmer groups and variants} \label{subsubsec:selmer-structure}
A Selmer structure $(\mathcal{F}, \Sigma)$ on $T$ consists of
\begin{itemize}
\item a finite set $\Sigma$ of places of $K$ containing the archimedean places, the primes where $T$ is ramified, and the prime $p$. 
\item a choice of submodule $\mathrm{H}^1_{\mathcal{F}}(K_v, T) \subseteq \mathrm{H}^1(K_v, T)$ for every place $v$ of $K$ such that
$\mathrm{H}^1_{\mathcal{F}}(K_v, T) = \mathrm{H}^1_{\mathrm{ur}}(K_v, T)$ for all $v \not\in \Sigma$.
\end{itemize}
 The \textbf{Selmer group $\mathrm{Sel}_{\mathcal{F}}(K, T)$ associated to $\mathcal{F}$} is defined by the exact sequence
 \[
\xymatrix{
0 \ar[r] & \mathrm{Sel}_{\mathcal{F}}(K, T) \ar[r] & \mathrm{H}^1(K, T) \ar[r] & \bigoplus_{v} \dfrac{\mathrm{H}^1(K_v, T)}{\mathrm{H}^1_{\mathcal{F}}(K_v, T)}
} 
 \]
where the direct sum runs over all places $v$ of $K$.
A Selmer structure $\mathcal{F}$ is self-dual if $\mathrm{H}^1_{\mathcal{F}}(K_v, T)$ is maximal isotropic under the (symmetric) local Tate pairing for every finite place $v$ of $K$.
Since $p > 2$, $\mathrm{H}^1(K_v , T) = 0$ for $v$ archimedean.

We assume that we are given a fixed self-dual Selmer structure throughout this section.

For a submodule or a quotient $S$ of $T$ and the Selmer structure $\mathcal{F}$ on $T$, there exists an induced Selmer structure on $S$ defined by using the natural map $\mathrm{H}^1(K_v, S) \to \mathrm{H}^1(K_v, T)$ or
$\mathrm{H}^1(K_v, T) \to \mathrm{H}^1(K_v, S)$, respectively. See \cite[Remark 2.1.3]{howard-bipartite} for details.

A rational prime $\ell$ is an \textbf{admissible prime for $(T, K)$} if
\begin{enumerate}
\item $\ell$ is coprime to the residue characteristic $p$,
\item $T$ is unramified at $\ell$,
\item $\ell$ is inert in $K/\mathbb{Q}$,
\item $\ell^2 \not\equiv 1 \pmod{p}$,
\item $a_\ell = \pm (\ell+1)$ in $R$
where $a_\ell$ is the trace of the arithmetic Frobenius at $\ell$ acting on $T$.
\end{enumerate}
When $R = \mathbb{Z}/p^k\mathbb{Z}$, $\ell$ is called a $k$-admissible prime, and write $\ell \in \mathcal{P}^{\mathrm{adm}}_k$.
Denote by $\mathcal{P}^{\mathrm{adm}}$ the set of admissible primes for $(T,K)$ and by $\mathcal{N}^{\mathrm{adm}}$ the set of square-free products of primes in $\mathcal{P}^{\mathrm{adm}}$.
By the definition of admissible primes, we have unique decomposition
$T \simeq R \oplus R(1)$
as $G_{K_\ell}$-modules where $\ell \in \mathcal{P}^{\mathrm{adm}}$.
By using the decomposition, we define the ordinary cohomology $\mathrm{H}^1_{\mathrm{ord}}(K_{\ell}, T)$ by
the image of $\mathrm{H}^1(K_{\ell}, R(1))$  in $\mathrm{H}^1(K_{\ell}, T)$ for each $\ell \in \mathcal{P}^{\mathrm{adm}}$.

\begin{lem}
For each $\ell \in \mathcal{P}^{\mathrm{adm}}$, the decomposition $T \simeq R \oplus R(1)$
 induces a decomposition
$$\mathrm{H}^1(K_{\ell}, T)  \simeq \mathrm{H}^1_{\mathrm{ur}}(K_{\ell}, T) \oplus \mathrm{H}^1_{\mathrm{ord}}(K_{\ell}, T)  $$
in which each summand is free of rank one over $R$ and is maximal isotropic under the local Tate pairing.
\end{lem}
\begin{proof}
See \cite[Lemma 2.2.1]{howard-bipartite}.
\end{proof}
A Selmer structure $\mathcal{F}$ is \textbf{cartesian} if for every quotient $T/\mathfrak{m}^i T$ of $T$, every place $v \in \Sigma$, and any generator $\pi \in \mathfrak{m}$, the isomorphism 
$$\times \pi^{\mathrm{length}(R) - i}: T/\mathfrak{m}^i T \to T[\mathfrak{m}^i]$$
induces an isomorphism
$$\mathrm{H}^1_{\mathcal{F}}(K_v, T/\mathfrak{m}^i T)
\simeq \mathrm{H}^1_{\mathcal{F}}(K_v, T[\mathfrak{m}^i]) .$$
\begin{assu} \label{assu:abs-irred-cartesian}
Throughout this section, we make the following assumptions:
\begin{enumerate}
\item The residual representation $T/\mathfrak{m}T$ is absolutely irreducible.
\item $\mathcal{F}$ is cartesian.
\end{enumerate}
\end{assu}
For any square-free integer $a \cdot b$ and any $c \in \mathcal{N}^{\mathrm{adm}}$, the Selmer structure $( \mathcal{F}^a_b(c) , \Sigma_{\mathcal{F}^a_b(c)} )$ is defined by
\begin{itemize}
\item $\Sigma_{\mathcal{F}^a_b(c)} = \Sigma \cup \left\lbrace q, \textrm{ a prime} : q \textrm{ divides } a \cdot b \cdot c \right\rbrace$,
\item $ \mathrm{H}^1_{\mathcal{F}^a_b(c)}(K_v, T) =  \mathrm{H}^1_{\mathcal{F}}(K_v, T)$ for $v$ prime to $a \cdot b \cdot c$, and
\item $\mathrm{H}^1_{\mathcal{F}^a_b(c)}(K_v, T)
=
\left\lbrace
\begin{array}{ll}
 \mathrm{H}^1(K_v, T) & \textrm{ if }  v \vert a \\
 0 & \textrm{ if } v \vert b \\
 \mathrm{H}^1_{\mathrm{ord}}(K_v, T) & \textrm{ if }  v \vert c
\end{array} \right.$.
\end{itemize}
If any one of $a$, $b$, $c$ is the empty product, we omit it from the notation.

\begin{lem}
The Selmer structure $\mathcal{F}(n)$ is cartesian for any $n \in \mathcal{N}^{\mathrm{adm}}$.
For any choice of generator $\pi \in \mathfrak{m}$ and any $0 \leq i \leq \mathrm{length}(R)$, the composition
\[
\xymatrix{
T/\mathfrak{m}^iT \ar[rr]^{\pi^{\times \mathrm{length}(R) - i}} & & T[\mathfrak{m}^i] \ar[r] & T
}
\]
induces isomorphisms
$$\mathrm{Sel}_{\mathcal{F}(n)}(K, T/\mathfrak{m}^iT) \simeq
\mathrm{Sel}_{\mathcal{F}(n)}(K, T[\mathfrak{m}^i]) \simeq
\mathrm{Sel}_{\mathcal{F}(n)}(K, T)[\mathfrak{m}^i].$$
\end{lem}
\begin{proof}
See \cite[Lemma 2.2.6]{howard-bipartite}.
\end{proof}
\begin{prop} \label{prop:selmer-group-structure-even-odd}
For any $n \in \mathcal{N}^{\mathrm{adm}}$, there is a non-canonical isomorphism
$$\mathrm{Sel}_{\mathcal{F}(n)}(K, T) \simeq R^{e(n)} \oplus M_n \oplus M_n$$
with $e(n) \in \lbrace 0, 1 \rbrace$.
\end{prop}
\begin{proof}
See \cite[Proposition 2.2.7]{howard-bipartite}.
\end{proof}
\begin{defn} \label{defn:stub-submodules}
Let $\mathcal{N}^{\mathrm{def}} \subseteq \mathcal{N}^{\mathrm{adm}}$ be the subset for which $e(n) = 0$ and
$\mathcal{N}^{\mathrm{ind}} \subseteq \mathcal{N}^{\mathrm{adm}}$ be the subset for which $e(n) = 1$.
For $n \in \mathcal{N}^{\mathrm{adm}}$, the stub submodule is defined by
$$\mathrm{Stub}_n  = \left \lbrace
    \begin{array}{ll}
     \mathfrak{m}^{\mathrm{length} (M_n)} \cdot R  & \textrm{ when } n \in \mathcal{N}^{\mathrm{def}} , \\ 
     \mathfrak{m}^{\mathrm{length} (M_n)} \cdot \mathrm{Sel}_{\mathcal{F}(n)}(K,T)  & \textrm{ when } n \in \mathcal{N}^{\mathrm{ind}} ,
    \end{array}
    \right. $$
with $M_n$ as in Proposition \ref{prop:selmer-group-structure-even-odd}. Note that $\mathrm{Stub}_n$ is a cyclic $R$-module for every $n \in \mathcal{N}^{\mathrm{adm}}$.
\end{defn}

\begin{prop} \label{prop:relation-lengths}
Let $n\ell \in \mathcal{N}^{\mathrm{adm}}$.
There exist non-negative integers $a$, $b$ with $a +b = \mathrm{length}(R)$ such that
the diagram of inclusions
\[
\xymatrix{
& \mathrm{Sel}_{\mathcal{F}^{\ell}(n)}(K, T) \\
\mathrm{Sel}_{\mathcal{F}(n)}(K, T) \ar@{^{(}->}[ur]^-{b} & & \mathrm{Sel}_{\mathcal{F}(n\ell)}(K, T) \ar@{_{(}->}[ul]_-{a} \\
& \mathrm{Sel}_{\mathcal{F}_{\ell}(n)}(K, T) \ar@{^{(}->}[ur]^-{b} \ar@{_{(}->}[ul]_-{a}
}
\]
where the labels on the arrows are the length of the respective quotients, i.e.
\begin{itemize}
\item $a = \mathrm{length} \left( \dfrac{ \mathrm{Sel}_{\mathcal{F}^{\ell}(n)}(K, T) }{ \mathrm{Sel}_{\mathcal{F}(n\ell)}(K, T) } \right) = \mathrm{length} \left( \dfrac{ \mathrm{Sel}_{\mathcal{F}(n)}(K, T) }{ \mathrm{Sel}_{\mathcal{F}_{\ell}(n)}(K, T) } \right)$.
\item $b = \mathrm{length} \left( \dfrac{ \mathrm{Sel}_{\mathcal{F}^{\ell}(n)}(K, T) }{ \mathrm{Sel}_{\mathcal{F}(n)}(K, T) } \right) = \mathrm{length} \left( \dfrac{ \mathrm{Sel}_{\mathcal{F}(n\ell)}(K, T) }{ \mathrm{Sel}_{\mathcal{F}_{\ell}(n)}(K, T) } \right)$.
\end{itemize}
All four quotients are cyclic $R$-modules and
\[
\xymatrix{
a = \mathrm{length} \left(
\mathrm{loc}_\ell \left( \mathrm{Sel}_{\mathcal{F}(n)}(K, T) \right)
 \right), &
b = \mathrm{length} \left(
\mathrm{loc}_\ell \left( \mathrm{Sel}_{\mathcal{F}(n\ell)}(K, T) \right)
 \right) .
}
\]
\end{prop}
\begin{proof}
See \cite[Proposition 2.2.9]{howard-bipartite}.
\end{proof}
\begin{cor}
Fix $n \in \mathcal{N}^{\mathrm{adm}}$ and $e(n)$ be as in Proposition \ref{prop:selmer-group-structure-even-odd}.
The integer
$$\rho(n) = \mathrm{dim}_{R/\mathfrak{m}} \left(  \mathrm{Sel}_{\mathcal{F}(n)}(K, T/\mathfrak{m}T) \right)$$
satisfies
$e(n ) \equiv  \rho(n) \pmod{2}$, and for any $\ell \in \mathcal{P}^{\mathrm{adm}}$ primes to $n$,
\begin{align*}
\rho(n\ell) & = \rho(n) +1 \Leftrightarrow \mathrm{loc}_\ell \left( \mathrm{Sel}_{\mathcal{F}(n)}(K, T/\mathfrak{m}T) \right) = 0, \\
\rho(n\ell) & = \rho(n) -1 \Leftrightarrow \mathrm{loc}_\ell \left( \mathrm{Sel}_{\mathcal{F}(n)}(K, T/\mathfrak{m}T) \right) \neq 0.
\end{align*}
In particular, for $n\ell \in \mathcal{N}^{\mathrm{adm}}$, $n \in \mathcal{N}^{\mathrm{def}}$ if and only if $n\ell \in \mathcal{N}^{\mathrm{ind}}$.
\end{cor}
\begin{proof}
See \cite[Corollary 2.2.10 and Remark 2.2.11]{howard-bipartite}.
\end{proof}
\begin{cor}
Suppose $n \ell \in \mathcal{N}^{\mathrm{adm}}$, and let $a$ and $b$ be as in Proposition \ref{prop:relation-lengths}.
Then
$$\mathrm{length}(M_n)   = \left \lbrace
    \begin{array}{ll}
    \mathrm{length}(M_{n\ell})  +a & \textrm{ when } n \in \mathcal{N}^{\mathrm{def}} , \\ 
     \mathrm{length}(M_{n\ell})  - b  & \textrm{ when } n \in \mathcal{N}^{\mathrm{ind}} .
    \end{array}
    \right.$$
\end{cor}
\begin{proof}
See \cite[Corollary 2.2.12]{howard-bipartite}.
\end{proof}
\begin{cor}
Suppose $n\ell \in \mathcal{N}^{\mathrm{adm}}$. There exists an isomorphism of $R$-modules
$$
    \begin{array}{ll}
    \mathrm{loc}_\ell(\mathrm{Stub}_n) \simeq \mathrm{Stub}_{n\ell} & \textrm{ if } n \in \mathcal{N}^{\mathrm{ind}} , \\ 
     \mathrm{loc}_\ell(\mathrm{Stub}_{n\ell}) \simeq \mathrm{Stub}_{n}  & \textrm{ if } n \in \mathcal{N}^{\mathrm{def}} .
    \end{array}
$$
\end{cor}
\begin{proof}
See \cite[Corollary 2.2.13]{howard-bipartite}.
\end{proof}
\subsubsection{Bipartite Euler systems}
We keep Assumption \ref{assu:abs-irred-cartesian} as well as the following one.
\begin{assu} \label{assu:chebotarev-bipartite}
For any non-zero $c \in \mathrm{H}^1(K, T/\mathfrak{m}T)$, there are infinitely many primes $\ell \in \mathcal{P}^{\mathrm{adm}}$
such that $\mathrm{loc}_\ell (c) \neq 0$.
\end{assu}
\begin{defn} \label{defn:bipartite-euler-systems}
A \textbf{bipartite Euler system for $(T, \mathcal{F}, \mathcal{P}^{\mathrm{adm}})$} is a pair of families
\[
\xymatrix{
\ks^{\mathrm{bip}} = \left\lbrace \kappa^{\mathrm{bip}}_n \in \mathrm{Sel}_{\mathcal{F}(n)}(K,T) : n \in \mathcal{N}^{\mathrm{ind}} \right\rbrace , & 
\lambdab^{\mathrm{bip}} = \left\lbrace \lambda^{\mathrm{bip}}_n \in R : n \in \mathcal{N}^{\mathrm{def}}  \right\rbrace
}
\]
related by the first and second explicit reciprocity laws:
\begin{enumerate}
\item For $n\ell \in \mathcal{N}^{\mathrm{ind}}$, there exists an isomorphism of $R$-modules
$$R / \lambda^{\mathrm{bip}}_n R \simeq \mathrm{H}^1_{\mathrm{ord}}(K_\ell, T) / R \cdot \mathrm{loc}_\ell(\kappa^{\mathrm{bip}}_{n\ell} ) .$$ 
\item For $n\ell \in \mathcal{N}^{\mathrm{def}}$, there exists an isomorphism of $R$-modules
$$R / \lambda^{\mathrm{bip}}_{n\ell} R \simeq \mathrm{H}^1_{\mathrm{ur}}(K_\ell, T) / R \cdot \mathrm{loc}_\ell(\kappa^{\mathrm{bip}}_{n} ) .$$ 
\end{enumerate}
\end{defn}
\begin{rem}
A bipartite Euler system is \textbf{non-trivial} if $\lambda^{\mathrm{bip}}_n \neq 0$ for some $n \in \mathcal{N}^{\mathrm{def}}$. 
Indeed, this definition of bipartite Euler systems is ``of odd type" following the original definition. Since any even type bipartite Euler system is always trivial \cite[Proposition 2.3.4]{howard-bipartite}, we ignore the odd-even type classification.

\end{rem}

\begin{lem} \label{lem:chebotarev-bipartite}
For any $n \in \mathcal{N}^{\mathrm{adm}}$ and any cyclic $R$-submodule $C \subseteq \mathrm{Sel}_{\mathcal{F}(n)}(K, T)$, there exist infinitely many $\ell \in \mathcal{P}^{\mathrm{adm}}$ such that the map
$\mathrm{loc}_\ell : C \to \mathrm{H}^1_{\mathrm{ur}}(K, T)$ is injective.
If $C$ is free of rank one over $R$, then for any such $\ell$, we have isomorphism $\mathrm{loc}_\ell : C \simeq \mathrm{H}^1_{\mathrm{ur}}(K, T)$.
\end{lem}
\begin{proof}
See \cite[Lemma 2.3.3]{howard-bipartite}. This is based on Assumption \ref{assu:chebotarev-bipartite}.
\end{proof}

\begin{prop}
Fix a bipartite Euler system $(\lambdab^{\mathrm{bip}}, \ks^{\mathrm{bip}})$ for $(T, \mathcal{F}, \mathcal{P}^{\mathrm{adm}})$.
Let $k = \mathrm{length}(R)$ and $M_n$ be as in Proposition \ref{prop:selmer-group-structure-even-odd}.
\begin{enumerate}
\item If $\lambda^{\mathrm{bip}}_n \neq 0$ for some $n \in \mathcal{N}^{\mathrm{def}}$, then $\mathfrak{m}^{k-1} \cdot M_n = 0$.
\item If $\kappa^{\mathrm{bip}}_n \neq 0$ for some $n \in \mathcal{N}^{\mathrm{ind}}$, then $\mathfrak{m}^{k-1} \cdot M_n = 0$.
\end{enumerate}
\end{prop}
\begin{proof}
See \cite[Proposition 2.3.5]{howard-bipartite}.
\end{proof}
\begin{defn} \label{defn:freeness-bipartite}
We say that  a bipartite Euler system $(\lambdab^{\mathrm{bip}}, \ks^{\mathrm{bip}})$ for $(T, \mathcal{F}, \mathcal{P}^{\mathrm{adm}})$
is \textbf{free} if for every $n \in \mathcal{N}^{\mathrm{ind}}$, there exists a free $R$-submodule 
$C_n \subseteq \mathrm{Sel}_{\mathcal{F}(n)}(K, T)$
 of rank one containing $\kappa^{\mathrm{bip}}_n$.
\end{defn}
\begin{thm} \label{thm:stub-submodulues-bipartite}
For any free bipartite Euler system $(\lambdab^{\mathrm{bip}}, \ks^{\mathrm{bip}})$ for $(T, \mathcal{F}, \mathcal{P}^{\mathrm{adm}})$, we have
\begin{itemize}
\item $\lambda^{\mathrm{bip}}_n \in \mathrm{Stub}_n$ for every $n \in \mathcal{N}^{\mathrm{def}}$, and
\item $\kappa^{\mathrm{bip}}_n \in \mathrm{Stub}_n$ for every $n \in \mathcal{N}^{\mathrm{ind}}$.
\end{itemize}
Equivalently, the $R$-module $M_n$ of Proposition \ref{prop:selmer-group-structure-even-odd} satisfies
$$\mathrm{length}(M_n)   \leq \left \lbrace
    \begin{array}{ll}
    \mathrm{length}(M_{n\ell})  +a & \textrm{ when } n \in \mathcal{N}^{\mathrm{def}} , \\ 
     \mathrm{length}(M_{n\ell})  - b  & \textrm{ when } n \in \mathcal{N}^{\mathrm{ind}} .
    \end{array}
    \right.$$
\end{thm}
\begin{proof}
See \cite[Theorem 2.3.7]{howard-bipartite}.
\end{proof}
\begin{thm} \label{thm:rigidity-bipartite}
We make the following assumptions:
\begin{enumerate}
\item The residual representation $T/\mathfrak{m}T$ is absolutely irreducible.
\item $\mathcal{F}$ is cartesian.
\item For any non-zero $c \in \mathrm{H}^1(K, T/\mathfrak{m}T)$, there exist infinitely many $\ell \in \mathcal{P}^{\mathrm{adm}}$ such that
$\mathrm{loc}_\ell(c) \neq 0$.
\end{enumerate}
Suppose that we are given a non-trivial free bipartite Euler system for $(T, \mathcal{F}, \mathcal{P}^{\mathrm{adm}} )$.
Then there exists a unique integer $\delta$, independently of $n \in \mathcal{N}^{\mathrm{adm}}$, with the property that
\begin{enumerate}
\item $\lambda^{\mathrm{bip}}_n$ generates $\mathfrak{m}^\delta \mathrm{Stub}_n$ for every $n \in \mathcal{N}^{\mathrm{def}}$, and
\item $\kappa^{\mathrm{bip}}_n$ also generates $\mathfrak{m}^\delta \mathrm{Stub}_n$ for every $n \in \mathcal{N}^{\mathrm{ind}}$.
\end{enumerate}
Furthermore, $\delta$ is given by
\begin{align*}
\delta & = \mathrm{min} \left\lbrace \mathrm{ind}(\lambda^{\mathrm{bip}}_n, R) : n \in \mathcal{N}^{\mathrm{def}} \right\rbrace \\
 & = \mathrm{min} \left\lbrace \mathrm{ind}(\kappa^{\mathrm{bip}}_n,  \mathrm{Sel}_{\mathcal{F}(n)}(K, T) ) : n \in \mathcal{N}^{\mathrm{ind}} \right\rbrace .
\end{align*}
\end{thm}
\begin{proof}
See \cite[Theorem 2.5.1]{howard-bipartite}.
\end{proof}

\subsubsection{The structure of Selmer groups}
Write $k = \mathrm{length}(R)  < \infty$ and 
$$\mathrm{Sel}_{\mathcal{F}}(K, T) \simeq\mathrm{Sel}_{\mathcal{F}^*}(K, T^*) \simeq M_1 \oplus M_1 $$
and
$$ M_1 \simeq \bigoplus_{i \geq 1} R/\mathfrak{m}^{d_i} R $$
with $d_1 \geq d_2 \geq \cdots$ by using Proposition \ref{prop:selmer-group-structure-even-odd}.
Here, $\mathcal{F}^*$ is the dual Selmer structure to $\mathcal{F}$ and $T^* = \mathrm{Hom}(T, \mu_{p^\infty})$.
As in Proposition \ref{prop:relation-lengths}, all stub submodules are cyclic $R$-modules.
Thus, Theorem \ref{thm:rigidity-bipartite} shows that
\begin{align*}
 \partial^{(r)}( \lambdab^{\mathrm{bip}} ) & = 
\mathrm{min} \left\lbrace \mathrm{ind} ( \lambda^{\mathrm{bip}}_n , R ) : n \in \mathcal{N}^{\mathrm{def}} \textrm{ with (even) } \nu(n) = r  \right\rbrace , \\
&
= \mathrm{min} \left\lbrace 
k, \delta + \mathrm{length}(M_n) : n \in \mathcal{N}^{\mathrm{def}} \textrm{ with (even) } \nu(n) = r
\right\rbrace .
\end{align*}
When $r =0$, we have equality
$$ \partial^{(0)}( \lambdab^{\mathrm{bip}} ) = 
\mathrm{min} \left\lbrace 
k, \delta + \mathrm{length}(M_1) \right\rbrace =
\mathrm{min} \left\lbrace 
k, \delta + \sum_{i\geq 1} d_i \right\rbrace .$$
By using Lemma \ref{lem:chebotarev-bipartite}, there exists a prime $\ell_1 \in \mathcal{P}^{\mathrm{adm}}$ such that
the restriction of the natural map $$\mathrm{loc}_{\ell_1} : \mathrm{Sel}_{\mathcal{F}}(K, T) \to \mathrm{H}^1_{\mathrm{ur}}(K_{\ell_1}, T) $$
to one of  the $R/\mathfrak{m}^{d_1} R$-components is injective.
Then we have exact sequence
\[
\xymatrix{
& \mathrm{Sel}_{\mathcal{F}(\ell_1)}(K, T) \\
0 \ar[r] & \mathrm{Sel}_{\mathcal{F}_{\ell_1}}(K, T) \ar[r] \ar@{^{(}->}[u] & \mathrm{Sel}_{\mathcal{F}}(K, T) \ar[r]^-{\mathrm{loc}_{\ell_1}} & \mathrm{H}^1_{\mathrm{ur}}(K_{\ell_1}, T) .
}
\]
By the choice of $\ell_1$, we have
$$\mathrm{Sel}_{\mathcal{F}_{\ell_1}}(K, T) \simeq \left( \bigoplus_{i \geq 2} R/\mathfrak{m}^{d_i} R \right) \oplus M_1 $$
Since $\mathrm{length} \left( \mathrm{loc}_{\ell_1} ( \mathrm{Sel}_{\mathcal{F}}(K, T) ) \right) = d_1$,
Proposition \ref{prop:relation-lengths} implies that
$$\mathrm{length} \left( \dfrac{ \mathrm{Sel}_{\mathcal{F}(\ell_1)}(K, T) }{  \mathrm{Sel}_{\mathcal{F}_{\ell_1}}(K, T)  } \right) = k- d_1.$$
Then by Proposition \ref{prop:selmer-group-structure-even-odd}, we have
$$\mathrm{Sel}_{\mathcal{F}(\ell_1)}(K, T) \simeq R \oplus M_{\ell_1} \oplus M_{\ell_1}$$
with $M_{\ell_1} \simeq  \bigoplus_{i \geq 2}R/\mathfrak{m}^{d_i} R$.
 
By using Lemma \ref{lem:chebotarev-bipartite} again, there exists a prime $\ell_2 \in \mathcal{P}^{\mathrm{adm}}$ such that
the restriction of the natural map
$$\mathrm{loc}_{\ell_2} : \mathrm{Sel}_{\mathcal{F}(\ell_1)}(K, T) \to \mathrm{H}^1_{\mathrm{ur}}(K_{\ell_2}, T) $$
to the $R$-component is an isomorphism.
Applying the same argument, we obtain exact sequence
\[
\xymatrix{
& \mathrm{Sel}_{\mathcal{F}(\ell_1\ell_2)}(K, T) \\
0 \ar[r] & \mathrm{Sel}_{\mathcal{F}(\ell_1)_{\ell_2}}(K, T) \ar[r] \ar[u]^-{\simeq} & \mathrm{Sel}_{\mathcal{F}(\ell_1)}(K, T) \ar[r]^-{\mathrm{loc}_{\ell_2}} & \mathrm{H}^1_{\mathrm{ur}}(K_{\ell_2}, T) \ar[r] & 0.
}
\]
Thus we have
$$\mathrm{Sel}_{\mathcal{F}(\ell_1\ell_2)}(K, T) \simeq  M_{\ell_1\ell_2} \oplus M_{\ell_1\ell_2}$$
with $M_{\ell_1\ell_2} \simeq  M_{\ell_1} \simeq  \bigoplus_{i \geq 2}R/\mathfrak{m}^{d_i} R $.
It shows that
$$ \partial^{(2)}( \lambdab^{\mathrm{bip}} ) = 
\mathrm{min} \left\lbrace 
k, \delta + \mathrm{length}(M_{\ell_1\ell_2}) \right\rbrace =
\mathrm{min} \left\lbrace 
k, \delta + \sum_{i\geq 2} d_i \right\rbrace .$$
By repeating this argument, we obtain that
$$ \partial^{(r)}( \lambdab^{\mathrm{bip}} )  =
\mathrm{min} \left\lbrace 
k, \delta + \sum_{i\geq r/2+1} d_i \right\rbrace $$
for even $r$. In particular, when $k \gg 0$, we recover $d_i$'s from $\lambdab^{\mathrm{bip}}$ as follows.
\begin{align*}
d_1 & =  \partial^{(0)} (\lambdab^{\mathrm{bip}}) - \partial^{(2)} (\lambdab^{\mathrm{bip}}) \\
d_2 &  =  \partial^{(2)} (\lambdab^{\mathrm{bip}}) - \partial^{(4)} (\lambdab^{\mathrm{bip}}) \\
d_3 &  =  \partial^{(4)} (\lambdab^{\mathrm{bip}}) - \partial^{(6)} (\lambdab^{\mathrm{bip}}) \\
 & \vdots \\
d_i &  =  \partial^{(2(i-1))} (\lambdab^{\mathrm{bip}}) - \partial^{(2i)} (\lambdab^{\mathrm{bip}})
\end{align*}
for $i \geq 1$.
To sum up, we have the following statement.
\begin{thm} \label{thm:structure-pk-bipartite}
Suppose that we are given a non-trivial free bipartite Euler system for $(T, \mathcal{F}, \mathcal{P}^{\mathrm{adm}} )$; in particular, $\lambdab^{\mathrm{bip}} \neq 0$.
Write $\mathrm{Sel}_{\mathcal{F}}(K,T) \simeq \bigoplus_{i \geq 1} \left( R/\mathfrak{m}^{d_i} R \right)^{\oplus 2}$
with $d_1 \geq d_2 \geq  \cdots$.
Then 
$$ \partial^{(r)}( \lambdab^{\mathrm{bip}} )  =
\mathrm{min} \left\lbrace 
k, \delta + \sum_{i\geq r/2+1} d_i \right\rbrace $$
for even $r \geq 0$.
\end{thm}
\begin{rem}
This is an analogue of \cite[Proposition 4.5.8]{mazur-rubin-book} for the bipartite Euler system setting.
\end{rem}

\subsection{Elliptic curves}
We keep all the working hypotheses in $\S$\ref{subsec:working-hypotheses} with odd $\nu(N^-)$.
\subsubsection{Applying $\S$\ref{subsec:formalism-bipartite} to elliptic curves}
Let $T$ be the $p$-adic Tate module of $E$ (cf. Remark \ref{rem:notation-T}).
We assume that $E[p]$ is absolutely irreducible as a $G_K$-module and $N^-$ is a square-free product of an odd number of primes.
We fix the self-dual Selmer structure $\mathcal{F} = \mathcal{F}_{\mathrm{cl}}$ as the classical Selmer structure.
Then for every $k \geq  1$, the triple $(T/p^kT, \mathcal{F}, \mathcal{P}^{\mathrm{adm}}_k)$ satisfies Assumptions \ref{assu:abs-irred-cartesian} and \ref{assu:chebotarev-bipartite} given in $\S$\ref{subsec:formalism-bipartite} thanks to \cite[Lemma 3.3.4]{howard-bipartite}.

For the actual construction of the bipartite Euler system for elliptic curves, the following assumptions are further needed.
\begin{assu} \label{assu:conditionCR}
\begin{enumerate}
\item $E[p]$ is absolutely irreducible as a $G_\mathbb{Q}$-module,
\item $E$ has a good reduction at $p$, and
\item If a prime $q$ divides $N^-$ and $q \equiv \pm 1 \pmod{p}$, then $E[p]$ is ramified at $q$.
\end{enumerate}
\end{assu}
\subsubsection{Weak level raising}
Let $f = \sum_{n \geq 1} a_n(f) q^n$ be the newform corresponding to $E$.
Let $B_{N^-}$ be the definite quaternion algebra over $\mathbb{Q}$ of discriminant $N^-$ and $R_{N^+}$ be an Eichler order of $B_{N^-}$ of level $N^+$.
By using the Jacquet--Langlands correspondence, we have the corresponding quaternionic automorphic form  
$$f_{N^-} : B^{\times}_{N^-} \backslash \widehat{B}^{\times}_{N^-} / \widehat{R}^{\times}_{N^+} \to \mathbb{Z}_p$$
 with normalization $f_{N^-} \not\equiv 0 \pmod{p}$ where $\widehat{(-)} = (-) \otimes \widehat{\mathbb{Z}}$. 

Let $n \in \mathcal{N}^{\mathrm{def}}_1$.
Let $\mathbb{T}^{N^-}(N^+)$ and $\mathbb{T}^{nN^-}(N^+)$ be the full Hecke algebras over $\mathbb{Z}_p$ faithfully acting on the $N^-$-new subspace of  $S_2(\Gamma_0(N), \mathbb{Z}_p)$ and the $nN^-$-new subspace of  $S_2(\Gamma_0(nN), \mathbb{Z}_p)$, respectively.
We define a $\mathbb{Z}_p$-algebra homomorphism
$$\pi_{f_{N^-}} : \mathbb{T}^{N^-}(N^+) \to \mathbb{Z}_p$$
by $\pi_{f_{N^-}}(T_q) = a_q(f)$ if $q \nmid N$ and $\pi_{f_{N^-}}(U_q) = a_q(f)$ if $q \mid N$.

Let $\mathfrak{m}_{f_{N^-}} \subseteq \mathbb{T}^{N^-}(N^+)$ be the maximal ideal corresponding to $f_{N^-}$.
\begin{thm}
We keep Assumption \ref{assu:conditionCR} so that $\mathfrak{m}_{f_{N^-}}$ is non-Eisenstein.
\begin{enumerate}
\item
For each $n \in \mathcal{N}^{\mathrm{def}}$,
there exists a surjective homomorphism
$$\pi_{f_{nN^-}} : \mathbb{T}^{nN^-}(N^+) \to \mathbb{Z}_p/I^{\mathrm{adm}}_n\mathbb{Z}_p$$
such that 
$\pi_{f_{nN^-}}(T_q) = \pi_{f_{N^-}}(T_q) =  a_q(f)$ if $q \nmid nN$, 
$\pi_{f_{nN^-}}(U_q) = \pi_{f_{N^-}}(U_q)  = a_q(f)$ if $q \mid N$ ,
and 
$\pi_{f_{nN^-}}(U_q) = \epsilon_q$ if $q \mid n$ where $\epsilon_q = \pm 1$ satisfying $a_q(f) = \epsilon_\ell \cdot (\ell+1)$ in $\mathbb{Z}_p/I^{\mathrm{adm}}_n\mathbb{Z}_p$.
\item For each $n \in \mathcal{N}^{\mathrm{def}}$,
there exists a mod $I^{\mathrm{adm}}_n$ quaternionic $\mathbb{T}^{nN^-}(N^+)$-eigenform
$$f_{nN^-} : B^{\times}_{nN^-} \backslash \widehat{B}^{\times}_{nN^-} / \widehat{R}^{\times}_{N^+} \to \mathbb{Z}_p/I^{\mathrm{adm}}_n\mathbb{Z}_p$$
such that the induced map on the Hecke algebra is $\pi_{f_{nN^-}}$ 
and $f_{nN^-} \not\equiv 0 \pmod{p}$.
\end{enumerate}
\end{thm}
\begin{proof}
See \cite[Theorems 5.15 and 9.3]{bertolini-darmon-imc-2005}. See \cite[Theorem 6.2]{pw-mu} how Assumption \ref{assu:conditionCR} is used.
See also \cite[Definition 4.1]{chida-hsieh-main-conj}.
\end{proof}
\subsubsection{Construction of $\lambdab^{\mathrm{bip}}$}
Let $n \in \mathcal{N}^{\mathrm{def}}_1$.
Let $H$ be the Hilbert class field of $K$ and identify $\mathrm{Gal}(H/K) \simeq K^{\times} \backslash \widehat{K}^\times / \widehat{\mathcal{O}}^\times_K$ via the global class field theory. An embedding $\Phi : K \hookrightarrow B$ induces map
$\widehat{\Phi}: K^{\times} \backslash \widehat{K}^\times / \widehat{\mathcal{O}}^\times_K \to B^{\times}_{nN^-} \backslash \widehat{B}^{\times}_{nN^-} / \widehat{R}^{\times}_{N^+}$ and then $\mathrm{Gal}(H/K) $ acts on $B^{\times}_{nN^-} \backslash \widehat{B}^{\times}_{nN^-} / \widehat{R}^{\times}_{N^+}$ by left translation with $\widehat{\Phi}$.
We define $\lambda^{\mathrm{bip}}_n$ by the toric period integral
$$\lambda^{\mathrm{bip}}_n = \sum_{\sigma \in \mathrm{Gal}(H/K)} f_{nN^-} ( \sigma \cdot \varsigma^{(0)} ) \in \mathbb{Z}_p/I^{\mathrm{adm}}_n\mathbb{Z}_p$$
where $\varsigma^{(0)}$ is the Gross point of conductor 1 reviewed in Appendix \ref{sec:gross-points}.

\subsubsection{Construction of $\ks^{\mathrm{bip}}$}
Let $n \in \mathcal{N}^{\mathrm{ind}}_1$. 
Since $\nu(N^-)$ is odd, we always have $\nu(n) \geq 1$. Choose one admissible prime $\ell$ dividing $n$ with $I^{\mathrm{adm}}_{n/\ell} = I^{\mathrm{adm}}_{\ell}$.
For an $I^{\mathrm{adm}}_{n/\ell}$-admissible form $f_{nN^-/\ell}$ and $I^{\mathrm{adm}}_{\ell}$-admissible prime $\ell$,
there exists a cohomology class
$\kappa_{f_{nN^-/\ell}}(\ell) \in \mathrm{H}^1(K, T/I^{\mathrm{adm}}_{n}T)$
coming from Heegner points on Shimura curve $X^{nN^-}(N^+)$.
See \cite[$\S$4.3]{chida-hsieh-main-conj} for the precise construction of $\kappa_{f_{nN^-/\ell}}(\ell)$.
More precisely, we have
$$\kappa_{f_{nN^-/\ell}}(\ell) \in \mathrm{Sel}_{\mathcal{F}_{\mathrm{cl}}(n)}(K, T/I^{\mathrm{adm}}_{n}T) $$
thanks to \cite[Proposition 4.7]{chida-hsieh-main-conj}, and we let
$$\kappa^{\mathrm{bip}}_n = \kappa_{f_{nN^-/\ell}}(\ell) .$$
For our purpose, we need to check the following statement.
\begin{lem}
The divisibility of $\kappa_{f_{nN^-/\ell}}(\ell)$ is independent of the choice of $\ell$.
\end{lem}
\begin{proof}
Let $q \in \mathcal{P}^{\mathrm{adm}}_1$ with $I^{\mathrm{adm}}_{n} = I^{\mathrm{adm}}_{nq}$.
By using the second explicit reciprocity law \cite[Theorem 4.2]{bertolini-darmon-imc-2005}, we have equality
$$\mathrm{loc}_q(\kappa_{f_{nN^-/\ell}}(\ell)) = \lambda^{\mathrm{bip}}_{nq}$$
under the isomorphism $\mathrm{H}^1_{\mathrm{ur}}(K_q, T/I^{\mathrm{adm}}_nT) \simeq \mathbb{Z}_p/I^{\mathrm{adm}}_{nq}\mathbb{Z}_p$.
Since $\lambda^{\mathrm{bip}}_{nq}$ itself is defined without using $\kappa_{f_{nN^-/\ell}}(\ell)$, 
the divisibility of $\lambda^{\mathrm{bip}}_{nq}$ is independent of the choice of $\ell$.
Thus, the conclusion follows.
\end{proof}

\subsection{Structure theorem}

Consider the following bipartite Euler system  for $(T/p^kT, \mathcal{F}_{\mathrm{cl}}, \mathcal{P}^{\mathrm{adm}}_k)$
\begin{align*}
\lambdab^{\mathrm{bip}, (k)} & = \left\lbrace \lambda^{\mathrm{bip}, (k)}_n = \lambda^{\mathrm{bip}}_n \pmod{p^k} \in \mathbb{Z}/p^k\mathbb{Z} : n \in \mathcal{N}^{\mathrm{def}}_k \right\rbrace , \\
\ks^{\mathrm{bip}, (k)} & = \left\lbrace \kappa^{\mathrm{bip}, (k)}_n = \kappa^{\mathrm{bip}}_n \pmod{p^k} \in \mathrm{H}^1_{\mathcal{F}_{\mathrm{cl}}(n)}(K, T/p^kT) : n \in \mathcal{N}^{\mathrm{ind}}_k \right\rbrace .
\end{align*}
By restricting $\mathcal{P}^{\mathrm{adm}}_k$ to $\mathcal{P}^{\mathrm{adm}}_j$ with $j \geq 2k$, we have a ``sub-"bipartite Euler system of the above
\begin{align*}
\lambdab^{\mathrm{bip}, (k),j} & = \left\lbrace \lambda^{\mathrm{bip}, (k)}_n = \lambda^{\mathrm{bip}}_n \pmod{p^k} \in \mathbb{Z}/p^k\mathbb{Z} : n \in \mathcal{N}^{\mathrm{def}}_j \right\rbrace , \\
\ks^{\mathrm{bip}, (k),j} & = \left\lbrace \kappa^{\mathrm{bip}, (k)}_n = \kappa^{\mathrm{bip}}_n \pmod{p^k} \in \mathrm{H}^1_{\mathcal{F}_{\mathrm{cl}}(n)}(K, T/p^kT) : n \in \mathcal{N}^{\mathrm{ind}}_j \right\rbrace .
\end{align*}

\begin{thm} \label{thm:structure-bipartite}
Suppose that $(\lambdab^{\mathrm{bip}, (k)}, \ks^{\mathrm{bip}, (k)})$ is free for every $k \geq 1$ and $\lambdab^{\mathrm{bip}} \neq 0$.
Then we have isomorphism
$$\mathrm{Sel}(K, E[p^\infty]) \simeq (\mathbb{Q}_p/\mathbb{Z}_p)^{\oplus \mathrm{ord}(\lambdab^{\mathrm{bip}})} \oplus \bigoplus_{i \geq 1} \left( \mathbb{Z}/p^{ \partial^{(\mathrm{ord}(\lambdab^{\mathrm{bip}}) + 2(i-1))} (\lambdab^{\mathrm{bip}}) - \partial^{(\mathrm{ord}(\lambdab^{\mathrm{bip}}) + 2i)} (\lambdab^{\mathrm{bip}}) }\mathbb{Z} \right)^{\oplus 2} .$$
\end{thm}

\begin{proof}
We closely follow \cite[Theorem 5.2.12]{mazur-rubin-book}.
Write
$$\mathrm{Sel}(K, E[p^\infty]) \simeq (\mathbb{Q}_p/\mathbb{Z}_p)^{\oplus r} \oplus \bigoplus_{i \geq r+1} \left( \mathbb{Z}/p^{ e_i}\mathbb{Z} \right)^{\oplus 2}$$
with $e_{r+1} \geq e_{r+2} \geq \cdots$. Let $e_1 = \cdots e_r = \infty$.
For $k > 0$, we have 
$$\mathrm{Sel}(K, E[p^k]) \simeq (\mathbb{Z}/p^{k}\mathbb{Z})^{\oplus r} \oplus \bigoplus_{i \geq r+1} \left( \mathbb{Z}/p^{ \mathrm{min} \lbrace k, e_i \rbrace }\mathbb{Z} \right)^{\oplus 2} .$$
By applying Theorem \ref{thm:structure-pk-bipartite}, we have
\begin{equation} \label{eqn:structure-bipartite}
\partial^{(2s)} (\lambdab^{\mathrm{bip}, (k)}) = \mathrm{min} \left\lbrace 
k, \delta + \sum_{i\geq s+1} \mathrm{min} \lbrace k, e_i \rbrace \right\rbrace 
\end{equation}
for every $s \geq 0$.

Fix $s  \geq 0$, and let 
$$h = \partial^{(2s)} (\lambdab^{\mathrm{bip}}) = \mathrm{min} \left\lbrace \mathrm{ind}(\lambda^{\mathrm{bip}}_n, \mathbb{Z}_p/I^{\mathrm{adm}}_n\mathbb{Z}_p) : n \in \mathcal{N}^{\mathrm{def}}_1 \textrm{ with } \nu(n) = 2s \right\rbrace.$$
Then there exists $n \in \mathcal{N}^{\mathrm{def}}_{1}$ (necessarily in $\mathcal{N}^{\mathrm{def}}_{h+1}$) with $\nu(n) = 2s$ such that
$\lambda^{\mathrm{bip}}_n \not\in p^{h+1}\mathbb{Z}_p/I^{\mathrm{adm}}_n \mathbb{Z}_p$, equivalently $\lambda^{\mathrm{bip},(h+1)}_n \neq 0 \in \mathbb{Z}_p/p^{h+1}\mathbb{Z}_p$.

Since the pair $(\lambdab^{\mathrm{bip}, (h+1)} , \ks^{\mathrm{bip}, (h+1)})$ is a free bipartite Euler system for $(T/p^{h+1}T, \mathcal{F}_{\mathrm{cl}}, \mathcal{P}^{\mathrm{adm}}_{h+1})$ again,
Theorem \ref{thm:stub-submodulues-bipartite} and Definition \ref{defn:stub-submodules} imply that
$\partial^{(2s)} (\lambdab^{\mathrm{bip}, (h+1)}) \leq h$. By (\ref{eqn:structure-bipartite}), it follows that
$$\partial^{(2s)} (\lambdab^{\mathrm{bip}, (k)}) \leq h = \partial^{(2s)} (\lambdab^{\mathrm{bip}})$$
for every $k \geq h+1$.
On the other hand, since 
$\lambda^{\mathrm{bip}}_n \in p^{h}\mathbb{Z}_p/I^{\mathrm{adm}}_n \mathbb{Z}_p$ for every $n \in \mathcal{N}^{\mathrm{def}}_{1}$ with $\nu(n) = 2 s$,
we have
$$\partial^{(2s)} (\lambdab^{\mathrm{bip}, (k),j}) \geq h = \partial^{(2s)} (\lambdab^{\mathrm{bip}})$$
for every $k \geq h$ with $j \geq 2k$.
Thus, we have
$\partial^{(2s)} (\lambdab^{\mathrm{bip}}) = \mathrm{sup} \left\lbrace \partial^{(2s)} (\lambdab^{\mathrm{bip}, (k),j}) : k \geq 1 \textrm{ with } j \geq 2k \right\rbrace$.
Also, $\partial^{(2s)} (\lambdab^{\mathrm{bip}, (k),j})$ is a non-decreasing function on $k$ with $j \geq 2k$ by (\ref{eqn:structure-bipartite}).
To sum up, we have
$$\partial^{(2s)} (\lambdab^{\mathrm{bip}}) = \lim_{k \to \infty} \partial^{(2s)} (\lambdab^{\mathrm{bip}, (k),j}) .$$
In the limit process, we may choose $j = 2k$ for every $k$.
The conclusion follows from this limit and (\ref{eqn:structure-bipartite}).
\end{proof}
The restriction of  $\mathcal{P}^{\mathrm{adm}}_k$ to $\mathcal{P}^{\mathrm{adm}}_j$ with $j \geq 2k$ yields a free bipartite Euler system for $T/p^kT$ for every $k \geq 1$.
\begin{lem} \label{lem:restriction-free}
For $j \geq 2k$, the pair
$(\lambdab^{\mathrm{bip}, (k),j}, \ks^{\mathrm{bip}, (k),j})$ forms a free bipartite Euler system for $(T/p^kT, \mathcal{F}_{\mathrm{cl}}, \mathcal{P}^{\mathrm{adm}}_j)$.
\end{lem}
\begin{proof}
See \cite[Lemma 3.3.6]{howard-bipartite}.
\end{proof}
By Lemma \ref{lem:restriction-free}, the following mod $p^k$ version follows easily from Theorem \ref{thm:structure-pk-bipartite}.
\begin{thm} \label{thm:structure-bipartite-non-free}
Let $k \geq 1$ be an integer, and j be an integer with $j \geq 2k$.
If  $\lambdab^{\mathrm{bip}, (k), j} \neq 0$, then we have isomorphism
\begin{align*}
& \mathrm{Sel}(K, E[p^k]) \\
& \simeq (\mathbb{Z}/p^k\mathbb{Z})^{\oplus \mathrm{ord}(\lambdab^{\mathrm{bip}, (k), j})} \oplus \bigoplus_{i \geq 1} \left( \mathbb{Z}/p^{ \partial^{(\mathrm{ord}(\lambdab^{\mathrm{bip}, (k), j}) + 2(i-1))} (\lambdab^{\mathrm{bip}, (k), j}) - \partial^{(\mathrm{ord}(\lambdab^{\mathrm{bip}, (k), j}) + 2i)} (\lambdab^{\mathrm{bip}, (k), j}) }\mathbb{Z} \right)^{\oplus 2} 
\end{align*}
where $\mathrm{ord}(\lambdab^{\mathrm{bip}, (k), j})  = \mathrm{min} \left\lbrace \nu(n) : n \in \mathcal{N}^{\mathrm{def}}_j , \lambda^{\mathrm{bip}, (k)}_n \neq 0 \right\rbrace$ and
\begin{align*}
k & > \partial^{(\mathrm{ord}(\lambdab^{\mathrm{bip}, (k), j})} (\lambdab^{\mathrm{bip}, (k), j}) - \partial^{(\mathrm{ord}(\lambdab^{\mathrm{bip}, (k), j}) + 2)} (\lambdab^{\mathrm{bip}, (k), j}) \\
& \geq \partial^{(\mathrm{ord}(\lambdab^{\mathrm{bip}, (k), j}) + 2)} (\lambdab^{\mathrm{bip}, (k), j}) - \partial^{(\mathrm{ord}(\lambdab^{\mathrm{bip}, (k), j}) + 4)} (\lambdab^{\mathrm{bip}, (k), j}) \\
& \geq \cdots .
\end{align*}
\end{thm}
\begin{rem}
\begin{enumerate}
\item Lemma \ref{lem:restriction-free} corresponds to the \emph{sufficiently liftable} Kolyvagin systems in \cite[Theorem 4.4.3]{mazur-rubin-book}. 
The standard Kolyvagin system argument does not require such a restriction if the core rank is one \cite[Theorem 4.4.1]{mazur-rubin-book}.
We expect that a similar result should exist for bipartite Euler systems and will investigate this aspect in a near future.
\item 
Since $\lambdab^{\mathrm{bip}}$ does not admit the action of complex conjugation (cf. \cite[$\S$6.1]{wei-zhang-mazur-tate}), we do not expect to observe the structure of each $\mathrm{Sel}(K, E[p^\infty])^{\pm}$ from $\lambdab^{\mathrm{bip}}$.
However, the corank of each $\mathrm{Sel}(K, E[p^\infty])^{\pm}$ can be still detected by considering the different variation of $\lambda^{\mathrm{bip}}_1$ along anticyclotomic Kolyvagin primes \cite{sweeting-kolyvagin}.
\end{enumerate}
\end{rem}

\section{Structural comparisons}
\subsection{Decomposition of Selmer groups}
Let $K$ be an imaginary quadratic field and $c \in \mathrm{Gal}(K/\mathbb{Q})$ the complex conjugation.
Since $\mathrm{Gal}(K/\mathbb{Q})$ has order 2 and $p > 3$ is odd, we have decomposition
\begin{equation} \label{eqn:decomposition}
\mathrm{Sel}(K, E[p^\infty]) = \mathrm{Sel}(K, E[p^\infty])^+ \oplus \mathrm{Sel}(K, E[p^\infty])^-
\end{equation}
where $\mathrm{Sel}(K, E[p^\infty])^\pm$ is the $c$-eigenspace with eigenvalue $\pm 1$, respectively.
Then we have isomorphisms
\begin{equation} \label{eqn:decomposition-sign}
\xymatrix{
\mathrm{Sel}(K, E[p^\infty])^+  \simeq \mathrm{Sel}(\mathbb{Q}, E[p^\infty]) , &
\mathrm{Sel}(K, E[p^\infty])^-  \simeq \mathrm{Sel}(\mathbb{Q}, E^K[p^\infty]) .
}
\end{equation}

\subsection{The structural Gross--Zagier formula} \label{subsec:structural-gross-zagier}
Suppose that $\nu(N^-)$ is even.
Recall (\ref{eqn:structure-selmer-before})
$$
\mathrm{Sel}(K, E[p^\infty])^{\pm}_{/\mathrm{div}} \simeq \bigoplus_{i \geq 1} \left( \mathbb{Z} / p^{a^{\pm}_i} \mathbb{Z} \right)^{\oplus 2}
$$
where $a^{\pm}_1 \geq  a^{\pm}_2 \geq  \cdots$.
\subsubsection{When $r^+(E/K) > r^-(E/K)$} \label{subsubsec:when-r+-r-}
We first assume that $W(E/\mathbb{Q}) \cdot (-1)^{\mathrm{ord}(\ks^{\mathrm{Heeg}}) +1} = +1$.
Under the working hypotheses in $\S$\ref{subsec:working-hypotheses}, by using Theorem \ref{thm:kolyvagin-vanishing-order}, we have
\begin{itemize}
\item $r^{+} (E/K)= \mathrm{ord}(\ks^{\mathrm{Heeg}}) +1$,
\item $\mathrm{ord}(\ks^{\mathrm{Heeg}}) - r^{-} (E/K) \geq 0$ and is even, and
\item $r^+(E/K) > r^-(E/K)$.
\end{itemize}
Applying Theorem \ref{thm:structure-kolyvagin}, we also have
\begin{align*}
a^{+}_i & = \partial^{(\mathrm{ord}(\ks^{\mathrm{Heeg}}) + 2i -1)} (\ks^{\mathrm{Heeg}}) - \partial^{(\mathrm{ord}(\ks^{\mathrm{Heeg}}) + 2i)} (\ks^{\mathrm{Heeg}} ) , \\
a^{-}_{i + (  \mathrm{ord}(\ks^{\mathrm{Heeg}}) - r^{ -}(E/K) )} & = \partial^{(\mathrm{ord}(\ks^{\mathrm{Heeg}}) + 2i -2)} (\ks^{\mathrm{Heeg}}) - \partial^{(\mathrm{ord}(\ks^{\mathrm{Heeg}}) + 2i-1)} (\ks^{\mathrm{Heeg}} )
\end{align*}
for $i \geq 1$.
In particular, (\ref{eqn:decomposition-sign}) implies
$$\mathrm{ord}(\kn(E))  = \mathrm{ord}(\ks^{\mathrm{Heeg}}) +1 ,$$
and the comparison between Theorem \ref{thm:structure-kn} and Theorem \ref{thm:structure-kolyvagin} via (\ref{eqn:decomposition}) implies 
\begin{align*}
& \partial^{(\mathrm{ord}(\kn(E)) + 2(i-1))} (\widetilde{\boldsymbol{\delta}}(E) ) - \partial^{(\mathrm{ord}(\kn(E)) + 2i)} (\widetilde{\boldsymbol{\delta}}(E) )  \\
& =
2  \cdot \left(\partial^{(\mathrm{ord}(\ks^{\mathrm{Heeg}}) + 2i -1)} (\ks^{\mathrm{Heeg}}) - \partial^{(\mathrm{ord}(\ks^{\mathrm{Heeg}}) + 2i)} (\ks^{\mathrm{Heeg}} ) \right)
\end{align*}
for $i \geq 1$.
If we further assume that $\mathrm{ord}(\ks^{\mathrm{Heeg}}) = r^{ -}(E/K) ) = \mathrm{ord}(\kn(E^K))$, we have equality
\begin{align*}
& \partial^{(\mathrm{ord}(\kn(E^K)) + 2(i-1))} (\widetilde{\boldsymbol{\delta}}(E^K) ) - \partial^{(\mathrm{ord}(\kn(E^K)) + 2i)} (\widetilde{\boldsymbol{\delta}}(E^K) )  \\
& =
2 \cdot \left(\partial^{(\mathrm{ord}(\ks^{\mathrm{Heeg}}) + 2i -2)} (\ks^{\mathrm{Heeg}}) - \partial^{(\mathrm{ord}(\ks^{\mathrm{Heeg}}) + 2i-1)} (\ks^{\mathrm{Heeg}} ) \right) 
\end{align*}
for $i \geq 1$.
The collection of all the above formulas can be regarded as a structural refinement of Gross--Zagier formula.
Theorem \ref{thm:main-higher-gross-zagier-easy} follows from these formulas and (\ref{eqn:corank-difference}).
\begin{rem} \label{rem:structural-perrin-riou}
Since $\kn(E)$ and $\kn(E^K)$ are the images of $ \ks^{\mathrm{Kato}}(E) $ and $ \ks^{\mathrm{Kato}}(E^K) $ under the dual exponential maps (as explained in Remark \ref{rem:main-higher-gross-zagier-easy}), respectively,  these formulas can be understood as a structural refinement of Perrin-Riou's conjecture.
\end{rem}

\subsubsection{When $r^-(E/K) > r^+(E/K)$}
We now assume that $W(E/\mathbb{Q}) \cdot (-1)^{\mathrm{ord}(\ks^{\mathrm{Heeg}}) +1} = -1$.
Applying the argument in $\S$\ref{subsubsec:when-r+-r-}, we obtain the exactly same result except the changes of sign and the role of $E$ and $E^K$.

\subsection{The structural Waldspurger formulas} \label{subsec:structural-waldspurger}
Suppose that $\nu(N^-)$ is odd.
Under the working hypotheses in $\S$\ref{subsec:working-hypotheses}, 
the comparison between Theorem \ref{thm:structure-kn} and Theorem \ref{thm:structure-bipartite} via (\ref{eqn:decomposition}) implies isomorphism
{ \scriptsize
\begin{align} \label{eqn:structural-waldspurger}
\begin{split}
& (\mathbb{Q}_p/\mathbb{Z}_p)^{\oplus \mathrm{ord}(\lambdab^{\mathrm{bip}})} \oplus \bigoplus_{i \geq 1} \left( \mathbb{Z}/p^{ \partial^{(\mathrm{ord}(\lambdab^{\mathrm{bip}}) + 2(i-1))} (\lambdab^{\mathrm{bip}}) - \partial^{(\mathrm{ord}(\lambdab^{\mathrm{bip}}) + 2i)} (\lambdab^{\mathrm{bip}}) }\mathbb{Z} \right)^{\oplus 2} \\
  \simeq \ &  \left( \mathbb{Q}_p/\mathbb{Z}_p \right)^{\oplus \mathrm{ord}(\kn(E))} \oplus \bigoplus_{j\geq 1} \left( \mathbb{Z}/p^{(  \partial^{(\mathrm{ord}(\kn(E)) + 2(j-1))} (\widetilde{\boldsymbol{\delta}}(E) ) - \partial^{(\mathrm{ord}(\kn(E)) + 2j)} (\widetilde{\boldsymbol{\delta}}(E) ) )/2}\mathbb{Z} \right)^{\oplus 2}  \\
&  \oplus
 \left( \mathbb{Q}_p/\mathbb{Z}_p \right)^{\oplus \mathrm{ord}(\kn)(E^K)} \oplus \bigoplus_{k\geq 1} \left( \mathbb{Z}/p^{(  \partial^{(\mathrm{ord}(\kn(E^K)) + 2(k-1))} (\widetilde{\boldsymbol{\delta}}(E^K) ) - \partial^{(\mathrm{ord}(\kn(E^K)) + 2k)} (\widetilde{\boldsymbol{\delta}}(E^K) ) )/2}\mathbb{Z} \right)^{\oplus 2}  ,
\end{split}
\end{align}
}
and the isomorphism itself can be regarded as a structural Waldspurger formula.
Theorem \ref{thm:main-higher-waldspurger-easy} immediately follows from  (\ref{eqn:structural-waldspurger}).

Let $k \geq 1$ be an integer and $j$ be also an integer with $j \geq 2k$.
We now assume that $\lambdab^{\mathrm{bip},(k),j}$ is non-trivial. Then we have the following formula
{ \scriptsize
\begin{align} \label{eqn:structural-waldspurger-mod-p-k}
\begin{split}
& (\mathbb{Z}/p^k\mathbb{Z})^{\oplus \mathrm{ord}(\lambdab^{\mathrm{bip},(k),j})} \oplus \bigoplus_{i_1 \geq 1} \left( \mathbb{Z}/p^{ \partial^{(\mathrm{ord}(\lambdab^{\mathrm{bip}, (k), j}) + 2(i_1-1))} (\lambdab^{\mathrm{bip}, (k), j}) - \partial^{(\mathrm{ord}(\lambdab^{\mathrm{bip}, (k), j}) + 2i_1)} (\lambdab^{\mathrm{bip}, (k), j}) }\mathbb{Z} \right)^{\oplus 2} \\
  \simeq \ &  \left( \mathbb{Z}/p^k\mathbb{Z} \right)^{\oplus \mathrm{ord}(\kn^{(k)}(E))} \oplus \bigoplus_{i_2\geq 1} \left( \mathbb{Z}/p^{(  \partial^{(\mathrm{ord}(\kn^{(k)}(E)) + 2(i_2-1))} (\kn^{(k)}(E) ) - \partial^{(\mathrm{ord}(\kn^{(k)}(E)) + 2i_2)} (\kn^{(k)}(E) ) )/2}\mathbb{Z} \right)^{\oplus 2}  \\
&  \oplus
 \left( \mathbb{Z}/p^k\mathbb{Z} \right)^{\oplus \mathrm{ord}(\kn^{(k)})(E^K)} \oplus \bigoplus_{i_3\geq 1} \left( \mathbb{Z}/p^{(  \partial^{(\mathrm{ord}(\kn^{(k)}(E^K)) + 2(i_3-1))} (\widetilde{\boldsymbol{\delta}}^{(k)}(E^K) ) - \partial^{(\mathrm{ord}(\kn^{(k)}(E^K)) + 2i_3)} (\kn^{(k)}(E^K) ) )/2}\mathbb{Z} \right)^{\oplus 2}  ,
\end{split}
\end{align}
}
If we further assume that
 $\kn^{(k)}(E)$ and $\kn^{(k)}(E^K)$ are non-zero,
then we have
{ \scriptsize
\begin{align} \label{eqn:main-higher-waldspurger-easy-mod-p-k}
\begin{split}
& \mathrm{ord} ( \lambdab^{\mathrm{bip},(k),j} ) =  \mathrm{ord} (\kn^{(k)}(E)) + \mathrm{ord} (\kn^{(k)}(E^K)) = \mathrm{cork}_{\mathbb{Z}/p^k\mathbb{Z}} \mathrm{Sel}(K, E[p^k]) , \\
& 2 \cdot \left( \partial^{( \mathrm{ord}( \lambdab^{\mathrm{bip},(k),j} ) )} ( \lambdab^{\mathrm{bip},(k),j} )
- \partial^{( \infty )} ( \lambdab^{\mathrm{bip},(k),j} ) \right) \\
& = 
\partial^{( \mathrm{ord} (\kn^{(k)}(E)) )} ( \kn^{(k)}(E) )
- \partial^{( \infty )} ( \kn^{(k)}(E) ) +
\partial^{( \mathrm{ord} (\kn^{(k)}(E^K)) )} ( \kn^{(k)}(E^K) )
- \partial^{( \infty )} ( \kn^{(k)}(E^K) )   .
\end{split}
\end{align}
}

\section{The non-triviality of Kolyvagin systems and the main conjectures} \label{sec:non-triviality-questions}
In this section, we investigate all the non-triviality questions mentioned in  \S\ref{subsec:non-triviality} following the strategy in \cite{kim-structure-selmer}.
More precisely, we relate  each non-triviality question to the localization of the corresponding  main conjecture at the augmentation ideal.

\subsection{The non-triviality of $\kn$}
We keep all the settings in $\S$\ref{subsec:working-hypotheses}.
Let $\mathbb{Q}_{\infty}$ be the cyclotomic $\mathbb{Z}_p$-extension of $\mathbb{Q}$ and $\Lambda^{\mathrm{cyc}} = \mathbb{Z}_p\llbracket \mathrm{Gal}(\mathbb{Q}_\infty/\mathbb{Q}) \rrbracket \simeq \mathbb{Z}_p\llbracket X \rrbracket$ the cyclotomic iwasawa algebra.

Let
$\mathrm{H}^1_{\mathrm{Iw}}(\mathbb{Q}, T) = \varprojlim_m \mathrm{H}^1(\mathbb{Q}_m, T)$ be the first Iwasawa cohomology group of $T$ over $\mathbb{Q}_{\infty}$ where $\mathbb{Q}_m$ is the cyclic subextension of $\mathbb{Q}$ of degree $p^m$ in $\mathbb{Q}_\infty$.
Let $\ks^{\mathrm{Kato}, \infty}$ be the $\Lambda^{\mathrm{cyc}}$-adic Kato's Kolyvagin system following the convention given in \cite{kim-structure-selmer}, and we have $\kappa^{\mathrm{Kato}, \infty}_1 \in \mathrm{H}^1_{\mathrm{Iw}}(\mathbb{Q}, T)$.
Denote by $\mathrm{Sel}_0(\mathbb{Q}_{\infty}, E[p^\infty] )$ the $p$-strict Selmer group of $E[p^\infty]$ over $\mathbb{Q}_\infty$.
Write $(-)^\vee = \mathrm{Hom}(-, \mathbb{Q}_p/\mathbb{Z}_p)$. 

In \cite{kim-structure-selmer}, the following statement is proved.
\begin{thm} \label{thm:non-triviality-kn}
Let $E$ be an elliptic curve over $\mathbb{Q}$ and $p \geq 5$ a prime such that $\overline{\rho}$ is surjective.
Let $\mathfrak{P}$ be a height one prime ideal of $\Lambda^{\mathrm{cyc}}$.
The following statements are equivalent.
\begin{enumerate}
\item $\ks^{\mathrm{Kato}, \infty}$ does not vanish modulo $\mathfrak{P}$.
\item The localization of  the  Iwasawa main conjecture  at $\mathfrak{P}$ holds; in other words,
$$
\mathrm{ord}_{\mathfrak{P}} \left( \mathrm{char}_{\Lambda^{\mathrm{cyc}}} \left( \dfrac{\mathrm{H}^1_{\mathrm{Iw}}(\mathbb{Q}, T)}{\Lambda \kappa^{\mathrm{Kato}, \infty}_1 }  \right) \right)
=
\mathrm{ord}_{\mathfrak{P}} \left( \mathrm{char}_{\Lambda^{\mathrm{cyc}}} \left( \mathrm{Sel}_0(\mathbb{Q}_{\infty}, E[p^\infty] )^\vee  \right) \right) .
$$
\end{enumerate}
\end{thm}
By applying Theorem \ref{thm:non-triviality-kn} to $\mathfrak{P} = X\Lambda^{\mathrm{cyc}}$, we easily obtain the following statement.
\begin{cor} \label{cor:non-triviality-kn}
We keep all the settings in $\S$\ref{subsec:working-hypotheses}.
Then 
the Iwasawa main conjecture localized at $X\Lambda^{\mathrm{cyc}}$ holds if and only if $\kn$ is non-zero.
\end{cor}
\begin{proof}
We give a sketch only here, and we refer to \cite{kim-structure-selmer} for details.
Let $$t = 
\left\lbrace
\begin{array}{ll}
 0 & \textrm{ if  $E$ has split multiplicative reduction at $p$,} \\
 \mathrm{length}_{\mathbb{Z}_p} \left( E(\mathbb{Q}_p)[p^\infty] \right)  & \textrm{ otherwise.}
\end{array} \right.$$
 and $p^t \cdot \kn = \left\lbrace p^t \cdot \widetilde{\delta}_n \in \mathbb{Z}_p/I^{\mathrm{cyc}}_n\mathbb{Z}_p \right\rbrace_{ n \in \mathcal{N}^{\mathrm{cyc}}_1 }$.
Then it is easy to see that $p^t \cdot \kn$ is non-zero if and only if $\kn$ is non-zero.
It suffices to show that $\ks^{\mathrm{Kato}}$ is non-trivial if and only if $p^t \cdot \kn$ is non-zero under our working hypotheses.
If $p^t \cdot \kn$ is identically zero, then $\ks^{\mathrm{Kato}}$ becomes a Kolyvagin system of core rank zero, so it is trivial.
If $\ks^{\mathrm{Kato}}$ is trivial, then $p^t \cdot \kn$ is identically zero since $p^t \cdot \kn$ occurs as the image of $\ks^{\mathrm{Kato}}$ under the dual exponential map.
\end{proof}
In particular, the Iwasawa main conjecture inverting $p$ is strong enough to obtain the non-triviality of $\kn$.
\subsection{The non-triviality of $\ks^{\mathrm{Heeg}}$}
We keep all the settings in $\S$\ref{subsec:working-hypotheses} with even $\nu(N^-)$.

Let $K_{\infty}$ be the anticyclotomic $\mathbb{Z}_p$-extension of $K$ and $\Lambda^{\mathrm{ac}} = \mathbb{Z}_p\llbracket \mathrm{Gal}(K_\infty/K) \rrbracket \simeq \mathbb{Z}_p\llbracket X \rrbracket$ the anticyclotomic iwasawa algebra.
Denote by $K_m$ the cyclic subextension of $K$ of degree $p^m$ in $K_\infty$. 
Let $\alpha$ be the unit root of $X^2 - a_p(E)X +p$ and $\kappa^{\mathrm{Heeg}, \alpha, \infty}_1$ the $p$-stabilized Heegner point over $K_\infty$ with $U_p$-eigenvalue $\alpha$ \cite{howard-kolyvagin}.
It is known that  $\kappa^{\mathrm{Heeg}, \alpha, \infty}_1 \neq 0$ thanks to the work of Cornut and Vatsal \cite{cornut-higher-heegner-points, vatsal-duke}.

Although we do not go through the precise definitions of various Selmer structures here, all the details can be found in \cite{howard-kolyvagin}.
\begin{thm}[Howard] \label{thm:howard-divisibility}
We keep all the assumptions in $\S$\ref{subsec:working-hypotheses} with even $\nu(N^-)$ and 
 assume that $E$ has good ordinary reduction at $p$.
 Then 
 $\widehat{\mathrm{Sel}}(K_\infty, T) = \varprojlim_m \mathrm{Sel}(K_m, T)$ and $\mathrm{Sel}( K_\infty, E[p^\infty])^\vee$ 
are $\Lambda^{\mathrm{ac}}$-modules of rank one, and there exists a finitely generated torsion $\Lambda^{\mathrm{ac}}$-module $M_\infty$ such that
\begin{enumerate}
\item $\mathrm{Sel}( K_\infty, E[p^\infty])^\vee \sim \Lambda \oplus M_{\infty} \oplus M_{\infty}$,
\item $\mathrm{char}_{\Lambda^{\mathrm{ac}}} \left( M_{\infty} \right) = \mathrm{char}_{\Lambda^{\mathrm{ac}}} \left( M_{\infty} \right)^\iota$, and
\item $\mathrm{char}_{\Lambda^{\mathrm{ac}}} \left(  \dfrac{\widehat{\mathrm{Sel}}(K_\infty, T)}{ \kappa^{\mathrm{Heeg}, \alpha, \infty}_1 } \right) \subseteq \mathrm{char}_{\Lambda^{\mathrm{ac}}} \left( M_\infty \right)$
\end{enumerate}
where $\sim$ means a pseudo-isomorphism of $\Lambda^{\mathrm{ac}}$-modules, $(-)^\iota : \Lambda^{\mathrm{ac}} \to \Lambda^{\mathrm{ac}}$ is the involution defined by $\gamma \mapsto \gamma^{-1}$ for $ \gamma \in \mathrm{Gal}(K_\infty/K)$.
\end{thm}
\begin{proof}
See \cite[Theorem 2.2.10]{howard-kolyvagin}.
\end{proof}
\begin{rem}
Perrin-Riou's Heegner point main conjecture \cite{perrin-riou-heegner} is the equality of (3) of Theorem \ref{thm:howard-divisibility}.
See \cite{wan-heegner, burungale-castella-kim, castella-wan-imc-derivatives, sweeting-kolyvagin} for the recent developments on this conjecture.
\end{rem}
\begin{thm} \label{thm:non-triviality-heegner}
We keep all the assumptions in $\S$\ref{subsec:working-hypotheses} with even $\nu(N^-)$ and assume that $E$ has good ordinary reduction at $p$.
Let $\mathfrak{P}$ be a height one prime of $\Lambda^{\mathrm{ac}}$.
Then the following statements are equivalent.
\begin{enumerate}
\item  $ \ks^{\mathrm{Heeg}, \alpha, \infty}$ is non-trivial modulo $\mathfrak{P}$.
\item $\mathrm{ord}_{\mathfrak{P}} \left( \mathrm{char}_{\Lambda^{\mathrm{ac}}} \left(  \dfrac{\widehat{\mathrm{Sel}}(K_\infty, T)}{ \kappa^{\mathrm{Heeg}, \alpha, \infty}_1 } \right) \right) = \mathrm{ord}_{\mathfrak{P}} \left( \mathrm{char}_{\Lambda^{\mathrm{ac}}} \left( M_\infty \right) \right)$.
\end{enumerate}
\end{thm}
\begin{proof}
Our argument mainly follows \cite[Theorem 2.2.10]{howard-kolyvagin} but with some modification.

Let $f_{\Lambda^{\mathrm{ac}}}$ be a generator of $\mathrm{char}_{\Lambda^{\mathrm{ac}}} \left(  \dfrac{\widehat{\mathrm{Sel}}(K_\infty, T)}{ \kappa^{\mathrm{Heeg}, \alpha, \infty}_1 }\right) $. 
We first suppose that $\mathfrak{P} \neq p\Lambda^{\mathrm{ac}}$ and write $\mathfrak{P} = (g)$ where $g \in \Lambda^{\mathrm{ac}}$.
Define $\mathfrak{P}_M  = (g+p^M)\Lambda^{\mathrm{ac}}$.
By taking $M$ sufficiently large, we may assume the following statements.
\begin{itemize}
\item $\mathfrak{P}_M$ is a height one prime ideal of  $\Lambda^{\mathrm{ac}}$.
\item $\Lambda^{\mathrm{ac}} / \mathfrak{P}_M \simeq \Lambda^{\mathrm{ac}} / \mathfrak{P}$ as rings.
\item $\mathfrak{P}_M$ and $f_{\Lambda^{\mathrm{ac}}}$ are relatively prime, so we have $\kappa^{\mathrm{Heeg}, \alpha, (\mathfrak{P}_M)}_1 \neq 0$.
\item Both natural maps
\[
\xymatrix{
\widehat{\mathrm{Sel}}(K_\infty, T)/  \mathfrak{P}_M \to \mathrm{Sel}(K, T_{\mathfrak{P}_M}),
& \mathrm{Sel}(K, ( T_{\mathfrak{P}_M} )^* )  \to \mathrm{Sel}(K_\infty, E[p^\infty])[\mathfrak{P}_M]
}
\]
have finite kernel and cokernel which are bounded by a constant depending only on $[S_{\mathfrak{P}_M} : \Lambda^{\mathrm{ac}} / \mathfrak{P}_M]$
but not on $\mathfrak{P}_M$ itself
where $S_{\mathfrak{P}_M}$ is the integral closure of $\Lambda^{\mathrm{ac}} / \mathfrak{P}_M$ \cite[Proposition 2.2.8]{howard-kolyvagin}.
\end{itemize}
Following the argument of \cite[Theorem 5.3.10]{mazur-rubin-book},
the equality 
$(\mathfrak{P}_M, \mathfrak{P}^n ) = 
(\mathfrak{P}_M, p^{Mn} )$
implies that
\begin{align*}
\mathrm{length}_{\mathbb{Z}_p} \left(  \mathrm{Sel}(K, T_{\mathfrak{P}_M} ) / S_{\mathfrak{P}_M} \kappa^{\mathrm{Heeg}, \alpha, (\mathfrak{P}_M)}_1 \right) & = 
\mathrm{length}_{\mathbb{Z}_p} \left(  \Lambda^{\mathrm{ac}} / (f_{\Lambda^{\mathrm{ac}}} , \mathfrak{P}_M \right) \\
& = \mathrm{length}_{\mathbb{Z}_p} \left(  \Lambda^{\mathrm{ac}} / (  \mathfrak{P}^{ \mathrm{ord}_{\mathfrak{P}} (  f_{\Lambda^{\mathrm{ac}}} )}, \mathfrak{P}_M \right) \\
& = M \cdot \mathrm{rk}_{\mathbb{Z}_p}  ( \Lambda^{\mathrm{ac}} / \mathfrak{P}_M )  \cdot \mathrm{ord}_{\mathfrak{P}} (  f_{\Lambda^{\mathrm{ac}}} )
\end{align*}
up to $O(1)$ as $M$ varies. Here, $O(1)$ means that the differences are constant independent of $M$.
Similarly, we have
\begin{align*}
2 \cdot \mathrm{length}_{\mathbb{Z}_p} (M_{\mathfrak{P}_M}) & = \mathrm{length}_{\mathbb{Z}_p} \left( \mathrm{Sel}(K,  ( T_{\mathfrak{P}_M} )^* )_{/\mathrm{div}} \right) \\
& = \mathrm{length}_{\mathbb{Z}_p} \left(  ( \mathrm{Sel}(K_\infty,  E[p^\infty])^\vee / \mathfrak{P}_M )_{\mathbb{Z}_p\textrm{-tors}} \right) \\
& = M \cdot \mathrm{rk}_{\mathbb{Z}_p}  ( \Lambda^{\mathrm{ac}} / \mathfrak{P}_M )  \cdot \mathrm{ord}_{\mathfrak{P}} \left( \mathrm{char}_{\Lambda^{\mathrm{ac}}} \left( ( \mathrm{Sel}( K_\infty, E[p^\infty])^\vee)_{\Lambda^{\mathrm{ac}}\textrm{-tors}}  \right) \right)
\end{align*}
up to $O(1)$ as $M$ varies where $M_{\mathfrak{P}_M} = M_\infty \otimes_{\Lambda^{\mathrm{ac}}} S_{\mathfrak{P}_M}$.
We also have equality
\begin{align*}
\mathrm{length}_{S_{\mathfrak{P}_M}} \left(  \mathrm{Sel}(K, T_{\mathfrak{P}_M} ) / S_{\mathfrak{P}_M} \kappa^{\mathrm{Heeg}, \alpha, (\mathfrak{P}_M)}_1 \right) = \partial^{(0)} ( \ks^{\mathrm{Heeg}, \alpha, (\mathfrak{P}_M)} ) .
\end{align*}

Although we consider Kolyvagin systems at height one prime $\mathfrak{P}_M$, the corresponding Selmer structure is self-dual thanks to \cite[Lemma 2.1.1]{howard-kolyvagin}. 
Since $\kappa^{\mathrm{Heeg}, \alpha, (\mathfrak{P}_M)}_1 \neq 0$, we have
\begin{align*}
\mathrm{cork}_{ S_{\mathfrak{P}_M} }\mathrm{Sel}(K, (T_{\mathfrak{P}_M})^*) & = 1, \\
\mathrm{Sel}(K, (T_{\mathfrak{P}_M})^*)_{/\mathrm{div}} & \simeq \bigoplus_{i \geq 1} \left( S_{\mathfrak{P}_M} /  \mathfrak{m}^{a_i}_{ S_{\mathfrak{P}_M} }  S_{\mathfrak{P}_M} \right)^{\oplus 2}
\end{align*}
where $\mathfrak{m}_{ S_{\mathfrak{P}_M} }$ is the maximal ideal of $S_{\mathfrak{P}_M}$ and $a_1 \geq  a_2 \geq  \cdots$ thanks to \cite[Theorem 1.6.1]{howard-kolyvagin}.
In particular, for any $k \geq 1$ such that $\kappa^{\mathrm{Heeg}, \alpha, (\mathfrak{P}_M)}_1 \neq 0$ in $S_{\mathfrak{P}_M} /\mathfrak{m}^{k}_{ S_{\mathfrak{P}_M}} S_{\mathfrak{P}_M}$, we have
\begin{align*}
\mathrm{Sel}(K, (T_{\mathfrak{P}_M})^*[\mathfrak{m}^{k}_{ S_{\mathfrak{P}_M}}])  & \simeq  S_{\mathfrak{P}_M} /\mathfrak{m}^{k}_{ S_{\mathfrak{P}_M}} S_{\mathfrak{P}_M} \oplus M^{(k)} \oplus M^{(k)}  \\
& \simeq S_{\mathfrak{P}_M} /\mathfrak{m}^{k}_{ S_{\mathfrak{P}_M}} S_{\mathfrak{P}_M} \oplus \bigoplus_{i \geq 1} \left( S_{\mathfrak{P}_M} /  \mathfrak{m}^{a_i}_{ S_{\mathfrak{P}_M} }  S_{\mathfrak{P}_M} \right)^{\oplus 2}
\end{align*}
with $\mathrm{length}_{S_{\mathfrak{P}_M}} M^{(k)} < k$ following the proof of \cite[Theorem 1.6.1]{howard-kolyvagin}.
Denote by $$\ks^{\mathrm{Heeg}, \alpha, (\mathfrak{P}_M), 2k-1} = \left\lbrace \kappa^{\mathrm{Heeg}, \alpha, (\mathfrak{P}_M)}_n : n \in \mathcal{N}^{\mathrm{ac}}_{2k-1} \right\rbrace$$
 the Heegner point Kolyvagin system for $(T_{\mathfrak{P}_M}/ \mathfrak{m}^{k}_{ S_{\mathfrak{P}_M}} T_{\mathfrak{P}_M}, \mathcal{F}_{\mathrm{cl}}, \mathcal{N}^{\mathrm{ac}}_{2k-1})$.
We recall the following rigidity result for Heegner point Kolyvagin systems \cite[Lemma 2.3.1]{zanarella-howard}\footnote{This is also studied in an unpublished manuscript of Stein--Weinstein \cite{stein-weinstein}.}, which is a strengthening of \cite[Lemma 1.6.4]{howard-kolyvagin}.
For $n \in \mathcal{N}^{\mathrm{ac}}_{2k-1}$, there exists a unique integer $\delta^{\mathrm{Heeg}}_{\mathfrak{P}_M}(k)$, independent of $n$, such that
\begin{equation} \label{eqn:rigidity-kolyvagin}
\left\langle \kappa^{\mathrm{Heeg}, \alpha, (\mathfrak{P}_M)}_n \right\rangle  = \mathfrak{m}^{\lambda^{(k)}(n) + \delta^{\mathrm{Heeg}}_{\mathfrak{P}_M}(k) }_{ S_{\mathfrak{P}_M}}\mathrm{Sel}_{\mathcal{F}_{\mathrm{cl}}(n)}(K, T_{\mathfrak{P}_M}/ \mathfrak{m}^{k}_{ S_{\mathfrak{P}_M}} T_{\mathfrak{P}_M})
\end{equation}
where $\lambda^{(k)}(n) = \mathrm{length}_{ S_{\mathfrak{P}_M} }  M^{(k)}(n)$ and
$\mathrm{Sel}_{\mathcal{F}_{\mathrm{cl}}(n)}(K, (T_{\mathfrak{P}_M})^*[\mathfrak{m}^{k}_{ S_{\mathfrak{P}_M}}])   \simeq  S_{\mathfrak{P}_M} /\mathfrak{m}^{k}_{ S_{\mathfrak{P}_M}} S_{\mathfrak{P}_M} \oplus M^{(k)}(n) \oplus M^{(k)}(n)$.
The restriction of Kolyvagin primes to $\mathcal{N}^{\mathrm{ac}}_{2k-1}$ is essential here and $\delta^{\mathrm{Heeg}}_{\mathfrak{P}_M}(k)$ is also independent of $k$ if $k \gg 0$.
By (\ref{eqn:rigidity-kolyvagin}), we have
\begin{align*}
\mathrm{length}_{S_{\mathfrak{P}_M}} M^{(k)} & = \sum_{i \geq 1} a_i \\
& = \partial^{(0)} ( \ks^{\mathrm{Heeg}, \alpha, (\mathfrak{P}_M)} ) - \delta^{\mathrm{Heeg}}_{\mathfrak{P}_M}(k) .
\end{align*}
Combining all the above computations, we obtain equalities
\begin{align} \label{eqn:heegner-non-triviality-1}
\begin{split}
& 2 \cdot M \cdot \mathrm{rk}_{\mathbb{Z}_p}  ( S_{\mathfrak{P}_M} )  \cdot \mathrm{ord}_{\mathfrak{P}} (  f_{\Lambda^{\mathrm{ac}}} ) \\
& = 2 \cdot \mathrm{length}_{\mathbb{Z}_p} \left(  \mathrm{Sel}(K, T_{\mathfrak{P}_M} ) / S_{\mathfrak{P}_M} \kappa^{\mathrm{Heeg}, \alpha, (\mathfrak{P}_M)}_1 \right) \\ 
& =  2 \cdot\dfrac{\mathrm{rk}_{\mathbb{Z}_p}  ( S_{\mathfrak{P}_M} ) }{e(S_{\mathfrak{P}_M} / \mathbb{Z}_p)} \cdot  \mathrm{length}_{S_{\mathfrak{P}_M}} \left(  \mathrm{Sel}(K, T_{\mathfrak{P}_M} ) / S_{\mathfrak{P}_M} \kappa^{\mathrm{Heeg}, \alpha, (\mathfrak{P}_M)}_1 \right) \\
& = 2 \cdot\dfrac{\mathrm{rk}_{\mathbb{Z}_p}  ( S_{\mathfrak{P}_M} ) }{e(S_{\mathfrak{P}_M} / \mathbb{Z}_p)} \cdot \partial^{(0)} ( \ks^{\mathrm{Heeg}, \alpha, (\mathfrak{P}_M)} ) \\
& \geq  2 \cdot\dfrac{\mathrm{rk}_{\mathbb{Z}_p}  ( S_{\mathfrak{P}_M} ) }{e(S_{\mathfrak{P}_M} / \mathbb{Z}_p)} \cdot \left( \partial^{(0)} ( \ks^{\mathrm{Heeg}, \alpha, (\mathfrak{P}_M)} ) -  \delta^{\mathrm{Heeg}}_{\mathfrak{P}_M}(k) \right) \\
& = \dfrac{\mathrm{rk}_{\mathbb{Z}_p}  ( S_{\mathfrak{P}_M} ) }{e(S_{\mathfrak{P}_M} / \mathbb{Z}_p)} \cdot  \mathrm{length}_{S_{\mathfrak{P}_M}} \left( \mathrm{Sel}(K, (T_{\mathfrak{P}_M})^*)_{/\mathrm{div}} \right)  \\
& =  \mathrm{length}_{\mathbb{Z}_p} \left( \mathrm{Sel}(K, (T_{\mathfrak{P}_M})^*)_{/\mathrm{div}} \right)
\end{split}
\end{align}
up to $O(1)$ as $M$ varies, and
\begin{align} \label{eqn:heegner-non-triviality-2}
\begin{split}
& \mathrm{length}_{\mathbb{Z}_p} \left( \mathrm{Sel}(K,  ( T_{\mathfrak{P}_M} )^* )_{/\mathrm{div}} \right) \\
& = M \cdot \mathrm{rk}_{\mathbb{Z}_p}  ( S_{\mathfrak{P}_M} )   \cdot \mathrm{ord}_{\mathfrak{P}} \left( \mathrm{char}_{\Lambda^{\mathrm{ac}}} \left( ( \mathrm{Sel}( K_\infty, E[p^\infty])^\vee)_{\Lambda^{\mathrm{ac}}\textrm{-tors}}  \right) \right)
\end{split}
\end{align}
up to $O(1)$ as $M$ varies again.

We prove (2) $\Rightarrow$ (1) first.
Suppose that
$$\mathrm{ord}_{\mathfrak{P}} \left( \mathrm{char}_{\Lambda^{\mathrm{ac}}} \left( ( \mathrm{Sel}( K_\infty, E[p^\infty])^\vee)_{\Lambda^{\mathrm{ac}}\textrm{-tors}}  \right) \right) 
=2  \cdot \mathrm{ord}_{\mathfrak{P}} (  f_{\Lambda^{\mathrm{ac}}} ) .$$
Combining (\ref{eqn:heegner-non-triviality-1}) with  (\ref{eqn:heegner-non-triviality-2}), 
the inequality above becomes an equality, so $\delta^{\mathrm{Heeg}}_{\mathfrak{P}_M}(k)$ is also a constant  as $M$ varies.
Since 
\begin{equation} \label{eqn:congruences-p-M}
\ks^{\mathrm{Heeg}, \alpha, (\mathfrak{P})} \equiv \ks^{\mathrm{Heeg}, \alpha, (\mathfrak{P}_M)} \pmod{p^M}
\end{equation}
for every $M \geq 1$,
we obtain $\delta^{\mathrm{Heeg}}_{\mathfrak{P}}(k) = \delta^{\mathrm{Heeg}}_{\mathfrak{P}_M}(k) <\infty$  by taking $M > k$, so $ \ks^{\mathrm{Heeg}, \alpha, (\mathfrak{P})}$ is also non-trivial.

Now we prove (1) $\Rightarrow$ (2).
Suppose that $ \ks^{\mathrm{Heeg}, \alpha, (\mathfrak{P})}$ is non-trivial, so $ \delta^{\mathrm{Heeg}}_{\mathfrak{P}}(k) < \infty$. 
By using the congruence (\ref{eqn:congruences-p-M}) again, $\delta^{\mathrm{Heeg}}_{\mathfrak{P}_M}(k)$ is bounded as $M$ varies, so the inequality above becomes an equality.
Combining (\ref{eqn:heegner-non-triviality-1}) with  (\ref{eqn:heegner-non-triviality-2}), we obtain
$$\mathrm{ord}_{\mathfrak{P}} \left( \mathrm{char}_{\Lambda^{\mathrm{ac}}} \left( ( \mathrm{Sel}( K_\infty, E[p^\infty])^\vee)_{\Lambda^{\mathrm{ac}}\textrm{-tors}}  \right) \right) 
=2  \cdot \mathrm{ord}_{\mathfrak{P}} (  f_{\Lambda^{\mathrm{ac}}} ) .$$

When $\mathfrak{P} = p\Lambda^{\mathrm{ac}}$, the same argument works by taking $\mathfrak{P}_M = X^M + p$.
\end{proof}

\begin{defn} \label{defn:non-anomalous}
For an elliptic curve $E$ with good ordinary reduction $p$, we say that \textbf{$E$ is non-anomalous at $p$} 
if
\begin{enumerate}
\item $a_p(E) \not\equiv 1 \pmod{p}$ if $p$ splits in $K$, and
\item $a_p(E) \not\equiv \pm 1 \pmod{p}$ if $p$ is inert in $K$.
\end{enumerate}
\end{defn}

\begin{cor} \label{cor:non-triviality-heegner}
We keep all the assumptions in $\S$\ref{subsec:working-hypotheses} with even $\nu(N^-)$ and 
 assume that $E$ has good ordinary reduction at $p$.
\begin{enumerate}
\item The Heegner point main conjecture localized at $X\Lambda^{\mathrm{ac}}$ holds if and only if  $ \ks^{\mathrm{Heeg}}$ is non-trivial, i.e. Kolyvagin's conjecture holds.
\item If $E$ is non-anomalous at $p$ and $\ks^{\mathrm{Heeg}}$ is non-trivial modulo $p$, then the Heegner point main conjecture holds.
\item  We assume that $E$ is non-anomalous at $p$ and $p$ splits in $K$. Then the following statements are equivalent.
\begin{enumerate}
\item $\ks^{\mathrm{Heeg}}$ is non-trivial modulo $p$.
\item  The Heegner point main conjecture holds and $p$ does not divide any Tamagawa factor at a prime dividing $N^+$.
\end{enumerate}
\item If the Heegner point main conjecture localized at $X\Lambda^{\mathrm{ac}}$ holds, then the following rank one $p$-converse results follow
\begin{align*}
\mathrm{cork}_{\mathbb{Z}_p}\mathrm{Sel}(K, E[p^\infty]) =1 & \Rightarrow \mathrm{ord}_{s=1}L(E/K,s) = 1, \\
\mathrm{cork}_{\mathbb{Z}_p}\mathrm{Sel}(\mathbb{Q}, E[p^\infty]) =1 & \Rightarrow \mathrm{ord}_{s=1}L(E,s) = 1. 
\end{align*}
\end{enumerate}
\end{cor}
\begin{proof}
\begin{enumerate}
\item 
We have
\begin{align*}
\ks^{\mathrm{Heeg}, \alpha} & = \left\lbrace \kappa^{\mathrm{Heeg}, \alpha}_n  \in \mathrm{H}^1(K, T/I^{\mathrm{ac}}_nT) : n \in \mathcal{N}^{\mathrm{ac}}_1 \right\rbrace \\
& = \left\lbrace C_{E,p} \cdot \kappa^{\mathrm{Heeg}}_n  \in \mathrm{H}^1(K, T/I^{\mathrm{ac}}_nT) : n \in \mathcal{N}^{\mathrm{ac}}_1 \right\rbrace \\
& =  C_{E,p} \cdot \ks^{\mathrm{Heeg}}_n 
\end{align*}
where $C_{E,p} \in \mathbb{Z}_p$ is the constant depending only on $E$ and the splitting behavior of $p$ in $K/\mathbb{Q}$.  See Remark \ref{rem:comparison-p-stabilization} below for further details.
If $E$ is non-anomalous at $p$, then $C_{E,p}$ is a $p$-adic unit.
Since $\partial^{(\infty)}( \ks^{\mathrm{Heeg}, \alpha}  ) = \partial^{(\infty)}( \ks^{\mathrm{Heeg}}  ) + \mathrm{ord}_p(C_{E,p})$,
 the non-trivialities of $\ks^{\mathrm{Heeg}}$ and $\ks^{\mathrm{Heeg}, \alpha}$ are equivalent.
The conclusion follows easily from Theorem \ref{thm:non-triviality-heegner}.
\item In the same manner, Theorem \ref{thm:non-triviality-heegner} implies that the mod $p$ non-triviality of $\ks^{\mathrm{Heeg}}$ $\Rightarrow$ the Heegner point main conjecture.
\item We first show  that (b) $\Rightarrow$ (a).
We follow the strategy of \cite{zanarella-howard} and give a sketch only.
Let $\chi$ be a finite order character on $\mathrm{Gal}(K_\infty/K)$ and $\mathfrak{P}_\chi \subseteq \Lambda^{\mathrm{ac}}$ the height one prime ideal defined by the kernel of $\chi: \Lambda^{\mathrm{ac}} \to \mathcal{O}_\chi$ where $\mathcal{O}_\chi$ is the discrete valuation ring generated by the values of $\chi$ over $\mathbb{Z}_p$.
Let $\ks^{\mathrm{Heeg}, \alpha, (\mathfrak{P}_\chi)}$ be the $\chi$-specialization of $\ks^{\mathrm{Heeg}, \alpha, \infty}$
and we may assume that $\kappa^{\mathrm{Heeg}, \alpha, (\mathfrak{P}_\chi)}_1 \neq 0$ by varying $\chi$.
Applying Kolyvagin's structure theorem (Theorem \ref{thm:structure-kolyvagin}) to this setting,
we obtain
$$\mathrm{length}_{\mathcal{O}_\chi} M_\chi = \partial^{(0)}( \ks^{\mathrm{Heeg}, \alpha, (\mathfrak{P}_\chi)} ) - \partial^{(\infty)}( \ks^{\mathrm{Heeg}, \alpha, (\mathfrak{P}_\chi)} ) $$
where $M_\chi = M_\infty \otimes_{  \Lambda^{\mathrm{ac}}, \chi  } \mathcal{O}_\chi$.
Note that we do not see any error term (``the $\chi$-twisted version of $\mathrm{Err}$" in Remark \ref{rem:error-term-heegner}) since $\kappa^{\mathrm{Heeg}, \alpha, (\mathfrak{P}_\chi)}_1 \neq 0$.
We now assume that the Heegner point main conjecture.
Since $p$ splits in $K$, it is equivalent to the BDP main conjecture \cite[Theorem 5.2]{burungale-castella-kim}. 
Evaluating the BDP $p$-adic $L$-function at $\chi^{-1}$, we obtain the following formula from \cite[Theorem 4.4.3]{zanarella-howard}, which is the $\chi$-twisted version of \cite[Proposition A.3]{burungale-castella-kim}
$$2 \cdot \mathrm{length}_{\mathcal{O}_\chi} M_\chi = 2 \cdot \partial^{(0)}( \ks^{\mathrm{Heeg}, \alpha, (\mathfrak{P}_\chi)} ) - \sum_{v \vert N^+} \mathrm{ord}_p\left(  c_v((T_{\mathfrak{P}_\chi})^*/K) \right)$$
where $c_v((T_{\mathfrak{P}_\chi})^*/K)$ is the Tamagawa factor of $(T_{\mathfrak{P}_\chi})^* = \mathrm{Hom}( T_{\mathfrak{P}_\chi} , \mu_{p^\infty} )  $ over $K$ at $v$ and  $v$ runs over places of $K$.
Thus, we obtain
$$2 \cdot \partial^{(\infty)}( \ks^{\mathrm{Heeg}, \alpha, (\mathfrak{P}_\chi)} ) = \sum_{v \vert N^+} \mathrm{ord}_p\left(  c_v((T_{\mathfrak{P}_\chi})^*/K) \right) .$$
Under our assumptions (especially on the non-anomalous one), \cite[Proposition 3.1.3]{zanarella-howard} shows that
$\partial^{(\infty)}( \ks^{\mathrm{Heeg}, \alpha} ) = 0$ if and only if 
$\partial^{(\infty)}( \ks^{\mathrm{Heeg}, \alpha, (\mathfrak{P}_\chi)} ) = 0$.
Also, \cite[Proposition 4.3.5]{zanarella-howard} implies that
$\sum_{v \vert N^+} \mathrm{ord}_p\left(  c_v(E/K) \right) = 0$ if and only if
$\sum_{v \vert N^+} \mathrm{ord}_p\left(  c_v((T_{\mathfrak{P}_\chi})^*/K) \right) = 0$.
Since the mod $p$ non-vanishing properties of $\ks^{\mathrm{Heeg}}$ and $\ks^{\mathrm{Heeg}, \alpha}$ are equivalent under the non-anomalous assumption,
$\ks^{\mathrm{Heeg}}$ is non-trivial modulo $p$.

For the direction (a) $\Rightarrow$ (b), we know that the mod $p$ non-triviality of $\ks^{\mathrm{Heeg}}$ $\Rightarrow$ the Heegner point main conjecture by using (2).
By using the same argument of the (b) $\Rightarrow$ (a) part, 
any Tamagawa factor at a prime dividing $N^+$ is not divisible by $p$ if and only if $\ks^{\mathrm{Heeg}}$ is not divisible by $p$. The conclusion follows.
\item
 The $p$-converse over $K$ immediately follows from (1), Theorem \ref{thm:kolyvagin-vanishing-order}.(1), and Gross--Zagier formula \cite{gross-zagier-original}.
 The $p$-converse over $\mathbb{Q}$ follows from  the $p$-converse over $K$ with a suitable choice of $K$ \cite{bump-friedberg-hoffstein, murty-murty-mean}.
 \end{enumerate}
\end{proof}
\begin{rem}
\begin{enumerate}
\item Corollary \ref{cor:non-triviality-heegner}.(1) says that a small piece of the Heegner point main conjecture is strong enough to deduce Kolyvagin's conjecture, 
i.e. the non-triviality of  $\ks^{\mathrm{Heeg}}$.
\item Corollary \ref{cor:non-triviality-heegner}.(2) significantly simplifies the proof of the main result of \cite{burungale-castella-kim}.
See also \cite{zanarella-howard}.
If we further assume that $\kappa^{\mathrm{Heeg}}_1 \neq 0$, then $p$ does not divide any Tamagawa factor at a prime dividing $N^+$ thanks to the work of Jetchev \cite{jetchev-global-divisibility}.
\item Corollary \ref{cor:non-triviality-heegner}.(3) is also proved in \cite{zanarella-howard}.
\end{enumerate}
\end{rem}

\begin{rem} \label{rem:supersingular}
We expect that most argument easily generalizes to the $\pm$-Iwasawa-theoretic setting when $E$ has supersingular reduction at $p$.
\end{rem}

\begin{rem} \label{rem:comparison-p-stabilization}
The standard $p$-stabilization process yields the following equality:
$$
\kappa^{\mathrm{Heeg}, \alpha}_n =
\left\lbrace
\begin{array}{ll}
 \left(1- \dfrac{1}{\alpha} \right)^2 \cdot \kappa^{\mathrm{Heeg}}_n & \textrm{ if }  p \textrm{ splits in } K \\
 \left(1- \dfrac{1}{\alpha^2} \right) \cdot \kappa^{\mathrm{Heeg}}_n & \textrm{ if } p \textrm{ is inert in } K .
\end{array} \right.$$
In \cite[$\S$2.3]{howard-kolyvagin},  Howard gave a construction of the $\Lambda^{\mathrm{ac}}$-adic Heegner point Kolyvagin system $\ks^{\mathrm{How}, \infty}$, and it is slightly different from the $p$-stabilized one.  His construction essentially follows Perrin-Riou's one in \cite{perrin-riou-heegner}.
By using  \cite[Lemma 2.3.3]{howard-kolyvagin},  we can deduce the following formula:
$$
\kappa^{\mathrm{How}}_n =
\left\lbrace
\begin{array}{ll}
 \left(1-a_p+p \right)^2 \cdot \kappa^{\mathrm{Heeg}}_n & \textrm{ if }  p \textrm{ splits in } K \\
 \left(1-a_p+p \right) \cdot \left(1+a_p+p \right) \cdot \kappa^{\mathrm{Heeg}}_n & \textrm{ if } p \textrm{ is inert in } K .
\end{array} \right.$$
where $\ks^{\mathrm{How}} = \left\lbrace  \kappa^{\mathrm{How}}_n \right\rbrace$ is the specialization of $\ks^{\mathrm{How}, \infty}$ at the trivial character.
Under the non-anomalous assumption,  all these ``$p$-adic multipliers" are $p$-adic units.
Since the ratio between these $p$-adic multipliers is also a $p$-adic unit,  any $\Lambda^{\mathrm{ac}}$-adic lift of the ratio is also $\Lambda^{\mathrm{ac}}$-adic unit.
Therefore, the difference between 
$\kappa^{\mathrm{Heeg}, \alpha, \infty}_1$ and $\kappa^{\mathrm{How}, \infty}_1$ is a $\Lambda^{\mathrm{ac}}$-adic unit.
\end{rem}

\subsection{The non-triviality of $\lambdab^{\mathrm{bip}}$}
We keep all the settings in $\S$\ref{subsec:working-hypotheses} with odd $\nu(N^-)$.
We also  assume that $E$ has good ordinary reduction at $p$ and  adapt most notation from the previous subsection.

Let $\lambda^{\mathrm{bip}, \alpha, \infty}_1 \in \Lambda^{\mathrm{ac}}$ be the $p$-stabilized Bertolini--Darmon's theta element for $E$ over $K_\infty$ \cite{bertolini-darmon-imc-2005}.
It is known that $\lambda^{\mathrm{bip}, \alpha, \infty}_1 \neq 0$ thanks to the work of Vatsal \cite{vatsal-uniform}.

In \cite[$\S$3.2]{howard-bipartite}, the mod $p^k$ version of $\Lambda^{\mathrm{ac}}$-adic bipartite Euler systems is illustrated.
As in the proof of Theorem \ref{thm:structure-bipartite}, $\Lambda^{\mathrm{ac}}$-adic bipartite Euler systems are defined by taking the inverse limit with respect to $k$.
Denote by $\lambdab^{\mathrm{bip}, \alpha, \infty}$ the definite part of the $\Lambda^{\mathrm{ac}}$-adic bipartite Euler system.

\begin{thm}[Bertolini--Darmon] \label{thm:bertolini-darmon-divisibility}
We keep all the assumptions in $\S$\ref{subsec:working-hypotheses} with odd $\nu(N^-)$ and 
 assume that $E$ has good ordinary reduction at $p$.
Then 
 $\mathrm{Sel}( K_\infty, E[p^\infty])^\vee$ 
is a finitely generated torsion $\Lambda^{\mathrm{ac}}$-module, and 
$$\left( \lambda^{\mathrm{bip}, \alpha, \infty}_1 \cdot \lambda^{\mathrm{bip}, \alpha, \infty, \iota}_1 \right) \subseteq \mathrm{char}_{\Lambda^{\mathrm{ac}}} \left( \mathrm{Sel}(K_\infty, E[p^\infty])^\vee \right) .$$
\end{thm}
\begin{proof}
See \cite{bertolini-darmon-imc-2005, pw-mu, chida-hsieh-main-conj, kim-pollack-weston}. See also \cite[Theorem 3.2.3.(b)]{howard-bipartite}.
\end{proof}
\begin{rem}
The anticyclotomic main conjecture for elliptic curves \`{a} la Bertolini--Darmon is the equality of the inequality in Theorem \ref{thm:bertolini-darmon-divisibility}.
\end{rem}
\begin{lem} \label{lem:theta-involution}
We keep all the assumptions in $\S$\ref{subsec:working-hypotheses} with odd $\nu(N^-)$ and 
 assume that $E$ has good ordinary reduction at $p$.
Then
$\left( \lambda^{\mathrm{bip}, \alpha, \infty}_1 \right) = \left( \lambda^{\mathrm{bip}, \alpha, \infty, \iota}_1 \right)$ in $\Lambda^{\mathrm{ac}}$.
\end{lem}
\begin{proof}
See \cite[Lemma 1.5]{bertolini-darmon-imc-2005}.
\end{proof}

\begin{thm} \label{thm:non-triviality-bipartite}
We keep all the assumptions in $\S$\ref{subsec:working-hypotheses} with odd $\nu(N^-)$ and 
 assume that $E$ has good ordinary reduction at $p$.
Let $\mathfrak{P}$ be a height one prime of $\Lambda^{\mathrm{ac}}$.
Then the following statements are equivalent.
\begin{enumerate}
\item  $ \lambdab^{\mathrm{bip}, \alpha, \infty}$ is non-trivial modulo $\mathfrak{P}$.
\item $2 \cdot \mathrm{ord}_{\mathfrak{P}} \left( \lambda^{\mathrm{bip}, \alpha, \infty}_1 \right) = \mathrm{ord}_{\mathfrak{P}} \left( \mathrm{char}_{\Lambda^{\mathrm{ac}}} \left( \mathrm{Sel}(K_\infty, E[p^\infty])^\vee \right) \right)$.
\end{enumerate}
\end{thm}
By Lemma \ref{lem:theta-involution}, the second statement is exactly the localization of anticyclotomic main conjecture for elliptic curves at $\mathfrak{P}$.
\begin{proof}
The direction (1) $\Rightarrow$ (2) is \cite[Theorem 3.2.3.(c)]{howard-bipartite}, so we focus on the converse.
Like the proof of Theorem \ref{thm:non-triviality-heegner}, our argument mainly follows \cite[$\S$3.4]{howard-bipartite} again but with some modification.

We first suppose that $\mathfrak{P} \neq p\Lambda^{\mathrm{ac}}$ and write $\mathfrak{P} = (g)$ where $g \in \Lambda^{\mathrm{ac}}$.
Define $\mathfrak{P}_M  = (g+p^M)\Lambda^{\mathrm{ac}}$.
By taking $M$ sufficiently large, we may assume the following statement.
\begin{itemize}
\item $\mathfrak{P}_M$ is a height one prime ideal of  $\Lambda^{\mathrm{ac}}$.
\item $\Lambda^{\mathrm{ac}} / \mathfrak{P}_M \simeq \Lambda^{\mathrm{ac}} / \mathfrak{P}$ as rings.
\item $\mathfrak{P}_M$ and $\lambda^{\mathrm{bip}, \alpha, \infty}_1$ are relatively prime, so $\lambda^{\mathrm{bip}, \alpha, (\mathfrak{P}_M)}_1 \neq 0$.
\item The natural map
$$\mathrm{Sel}(K, ( T_{\mathfrak{P}_M} )^* )  \to \mathrm{Sel}(K_\infty, E[p^\infty])[\mathfrak{P}_M]$$
has finite kernel and cokernel which are bounded by a constant depending only on $[S_{\mathfrak{P}_M} : \Lambda^{\mathrm{ac}} / \mathfrak{P}_M]$
but not on $\mathfrak{P}_M$ itself \cite[Proposition 3.3.1]{howard-bipartite}.
\end{itemize}
Following the argument of \cite[Theorem 5.3.10]{mazur-rubin-book}, we obtain
$$\mathrm{length}_{\mathbb{Z}_p}  \left( 
\mathrm{Sel}(K, ( T_{\mathfrak{P}_M} )^* )
\right) = M \cdot \mathrm{rk}_{\mathbb{Z}_p} S_{\mathfrak{P}_M} \cdot \mathrm{ord}_{\mathfrak{P}} \left( \mathrm{Sel}(K_\infty, E[p^\infty])^\vee \right) $$
up to $O(1)$ as $M$ varies.
We also have
$$\mathrm{length}_{\mathbb{Z}_p}  \left( 
S_{\mathfrak{P}_M} / S_{\mathfrak{P}_M} \cdot \lambda^{\mathrm{bip}, \alpha, (\mathfrak{P}_M)}_1
\right)
= M \cdot \mathrm{rk}_{\mathbb{Z}_p} S_{\mathfrak{P}_M}  \cdot 
\mathrm{ord}_{\mathfrak{P}} \left( \lambda^{\mathrm{bip}, \alpha, \infty}_1 \right)$$
up to $O(1)$ as $M$ varies.

For any pair of positive integers $k \leq j$, set
$$\delta_{\mathfrak{P}_M}(k,j) = \mathrm{min} \left\lbrace \mathrm{ind} (\lambda^{\mathrm{bip}, \alpha, (\mathfrak{P}_M) }_n , S_{\mathfrak{P}_M}/p^k S_{\mathfrak{P}_M}) : n \in \mathcal{N}^{\mathrm{def}}_j \right\rbrace  .$$
Since $\lambda^{\mathrm{bip}, \alpha, (\mathfrak{P}_M) }_1 \neq 0$ in $S_{\mathfrak{P}_M}$, $\delta_{\mathfrak{P}_M}(k,j) < \infty$ for sufficiently large $k$. 
As $\delta_{\mathfrak{P}_M}(k,j) \leq \delta_{\mathfrak{P}_M}(k,j+1)$, we define
$\delta_{\mathfrak{P}_M}(k) = \lim_{j \to \infty} \delta_{\mathfrak{P}_M}(k,j) $.
By applying \cite[Proposition 3.3.3]{howard-bipartite}, we have
$$ \mathrm{length}_{S_{\mathfrak{P}_M}}  \left( \mathrm{Sel}(K, ( T_{\mathfrak{P}_M} )^* ) \right) =  2 \cdot \left( \partial^{(0)} ( \lambdab^{\mathrm{bip}, \alpha, (\mathfrak{P}_M)} ) -\delta_{\mathfrak{P}_M}(k) \right) $$
for every $k \geq 1$ such that $\lambda^{\mathrm{bip}, \alpha, (\mathfrak{P}_M) }_1 \neq 0$ in $S_{\mathfrak{P}_M}/ p^k S_{\mathfrak{P}_M}$.
In particular, $\delta_{\mathfrak{P}_M}(k)$ is independent of $k$ if $k$ is sufficiently large.
For such an integer $k$, we have
\begin{align*}
& M \cdot \mathrm{rk}_{\mathbb{Z}_p} S_{\mathfrak{P}_M} \cdot \mathrm{ord}_{\mathfrak{P}} \left( \mathrm{Sel}(K_\infty, E[p^\infty])^\vee \right) \\
& = \mathrm{length}_{\mathbb{Z}_p}  \left(  \mathrm{Sel}(K, ( T_{\mathfrak{P}_M} )^* ) \right) \\
& = \dfrac{\mathrm{rk}_{\mathbb{Z}_p}  ( S_{\mathfrak{P}_M} ) }{e(S_{\mathfrak{P}_M} / \mathbb{Z}_p)} \cdot  \mathrm{length}_{S_{\mathfrak{P}_M}}  \left( \mathrm{Sel}(K, ( T_{\mathfrak{P}_M} )^* ) \right) \\
& = \dfrac{\mathrm{rk}_{\mathbb{Z}_p}  ( S_{\mathfrak{P}_M} ) }{e(S_{\mathfrak{P}_M} / \mathbb{Z}_p)} \cdot  2 \cdot \left( \partial^{(0)} ( \lambdab^{\mathrm{bip}, \alpha, (\mathfrak{P}_M)} ) - \delta_{\mathfrak{P}_M}(k) \right) \\
& \leq 2 \cdot \dfrac{\mathrm{rk}_{\mathbb{Z}_p}  ( S_{\mathfrak{P}_M} ) }{e(S_{\mathfrak{P}_M} / \mathbb{Z}_p)} \cdot \partial^{(0)} ( \lambdab^{\mathrm{bip}, \alpha, (\mathfrak{P}_M)} ) \\
& = 2 \cdot \dfrac{\mathrm{rk}_{\mathbb{Z}_p}  ( S_{\mathfrak{P}_M} ) }{e(S_{\mathfrak{P}_M} / \mathbb{Z}_p)} \cdot  \mathrm{length}_{S_{\mathfrak{P}_M}}  \left(  S_{\mathfrak{P}_M} / S_{\mathfrak{P}_M} \cdot \lambda^{\mathrm{bip}, \alpha, (\mathfrak{P}_M)}_1 \right) \\
& = 2 \cdot \mathrm{length}_{\mathbb{Z}_p}  \left( S_{\mathfrak{P}_M} / S_{\mathfrak{P}_M} \cdot \lambda^{\mathrm{bip}, \alpha, (\mathfrak{P}_M)}_1 \right) \\
& = 2 \cdot  M \cdot \mathrm{rk}_{\mathbb{Z}_p} S_{\mathfrak{P}_M}  \cdot \mathrm{ord}_{\mathfrak{P}} \left( \lambda^{\mathrm{bip}, \alpha, \infty}_1 \right)
\end{align*}
up to $O(1)$ as $M$ varies, and the inequality in the middle is an equality if and only if $\delta_{\mathfrak{P}_M}(k)$ is a constant as $M$ varies.

If we assume $2 \cdot \mathrm{ord}_{\mathfrak{P}} \left( \lambda^{\mathrm{bip}, \alpha, \infty}_1 \right) = \mathrm{ord}_{\mathfrak{P}} \left( \mathrm{char}_{\Lambda^{\mathrm{ac}}} \left( \mathrm{Sel}(K_\infty, E[p^\infty])^\vee \right) \right)$, then
the inequality in the middle becomes an equality, so $\delta_{\mathfrak{P}_M}(k)$ is a constant as $M$ varies.
Since
$$\lambdab^{\mathrm{bip}, \alpha, (\mathfrak{P})} \equiv \lambdab^{\mathrm{bip}, \alpha, (\mathfrak{P}_M)} \pmod{p^M} ,$$
we have $\delta_{\mathfrak{P}}(k) = \delta_{\mathfrak{P}_M}(k) <\infty$ by taking $M > k$.
Hence, $\lambdab^{\mathrm{bip}, \alpha, (\mathfrak{P})}$ is also non-trivial,
When $\mathfrak{P} = p\Lambda^{\mathrm{ac}}$, the same argument works by taking $\mathfrak{P}_M = X^M + p$.
\end{proof}

\begin{cor} \label{cor:non-triviality-bipartite}
We keep all the assumptions in $\S$\ref{subsec:working-hypotheses} with odd $\nu(N^-)$ and 
 assume that $E$ has good ordinary reduction at $p$.
The anticyclotomic main conjecture localized at $X\Lambda^{\mathrm{ac}}$ holds if and only if $\lambdab^{\mathrm{bip}}$ is non-trivial.
\end{cor}
\begin{proof}
We have
\begin{align*}
\lambdab^{\mathrm{bip}, \alpha} & = \left\lbrace \lambda^{\mathrm{bip}, \alpha}_n  \in \mathbb{Z}_p/I^{\mathrm{adm}} \mathbb{Z}_p : n \in \mathcal{N}^{\mathrm{def}}_1 \right\rbrace \\
& = \left\lbrace C'_{E,p} \cdot \lambda^{\mathrm{bip}}_n  \in \mathbb{Z}_p/I^{\mathrm{adm}} \mathbb{Z}_p : n \in \mathcal{N}^{\mathrm{def}}_1 \right\rbrace \\
& = C'_{E,p} \cdot \lambdab^{\mathrm{bip}}
\end{align*}
where $C'_{E,p} \in \mathbb{Z}_p$ is the constant depending only on $E$ and the splitting behavior of $p$ in $K/\mathbb{Q}$.
If $E$ is non-anomalous at $p$, then $C'_{E,p}$ is a $p$-adic unit.
Since $\partial^{(\infty)}( \lambdab^{\mathrm{bip}, \alpha} ) = \partial^{(\infty)}( \lambdab^{\mathrm{bip}}  ) + \mathrm{ord}_p(C'_{E,p})$,
 the non-trivialities of $\lambdab^{\mathrm{bip}}$ and $\lambdab^{\mathrm{bip}, \alpha}$ are equivalent.
The conclusion follows from Theorem \ref{thm:non-triviality-bipartite}.
\end{proof}

\begin{rem}
\begin{enumerate}
\item Theorem \ref{thm:non-triviality-bipartite} refines \cite[Theorem 3.2.3.(c)]{howard-bipartite}. 
In particular, under our working hypotheses, if $E$ is non-anomalous at $p$  and $\lambdab^{\mathrm{bip}}$ is non-trivial modulo $p$, then the anticyclotomic main conjecture for elliptic curves holds.
\item Remark \ref{rem:supersingular} applies to the bipartite Euler system setting equally.
\end{enumerate}
\end{rem}

\appendix

\section{Construction of the Gross point of conductor 1} \label{sec:gross-points}
We quickly review the construction of the Gross point of conductor 1, which is used for  the construction of $\lambda^{\mathrm{bip}}_n$ in $\S$\ref{sec:bipartite-euler-systems}. Some notation might be slightly different from those in the main text. See  \cite{chida-hsieh-main-conj, chida-hsieh-p-adic-L-functions, kim-overconvergent} for details.

Let $K$ be the imaginary quadratic field of discriminant $-D_K <0$.
Define
$$\vartheta := 
\left \lbrace
    \begin{array}{ll}
     \dfrac{ D_K - \sqrt{-D_K} }{ 2 }  & \textrm{ if } 2 \nmid D_K \\ 
\dfrac{ D_K - 2\sqrt{-D_K} }{ 4 }  & \textrm{ if } 2 \mid D_K
    \end{array}
    \right.
$$
 so that
$\mathcal{O}_K = \mathbb{Z} + \mathbb{Z}\vartheta$.
Let $B_{nN^-}$ be the definite quaternion algebra over $\mathbb{Q}$ of discriminant $nN^-$.
Then there exists an embedding of $K$ into $B_{nN^-}$ \cite{vigneras}.
More explicitly, we choose a $K$-basis $(1,J)$ of $B_{nN^-}$ so that $B_{nN^-} = K \oplus K \cdot J$ such that
$\beta := J^2 \in \mathbb{Q}^\times$ with $\beta <0$, $J \cdot t = \overline{t} \cdot J$ for all $t \in K$, $\beta \in \left( \mathbb{Z}^\times_q \right)^2$ for all $q \mid pN^+$, and
$\beta \in \mathbb{Z}^\times_q$ for all $q \mid D_K$.
Fix a square root $\sqrt{\beta} \in \overline{\mathbb{Q}}$ of $\beta$.
Fix an isomorphism
$$i := \prod i_q : \widehat{B}^{(nN^-)}_{nN^-} \simeq \mathrm{M}_2(\mathbb{A}^{(nN^-\infty)})$$
as follows.
\begin{itemize}
\item For each finite place $q \mid N^+p$, the isomorphism
$i_q : B_{nN^-,q} \simeq \mathrm{M}_2(\mathbb{Q}_q)$ is defined by
\[
\xymatrix{
{
i_q(\vartheta)  = \left( \begin{matrix}
\mathrm{trd}(\vartheta) & - \mathrm{nrd}(\vartheta) \\
1 & 0
\end{matrix} \right)  }, & 
{
i_q(J)  = \sqrt{\beta} \cdot \left( \begin{matrix}
-1 & \mathrm{trd}(\vartheta) \\
0 & 1
\end{matrix} \right) 
}
}
\]
where $\mathrm{trd}$ and $\mathrm{nrd}$ are the reduced trace and the reduced norm on $B$, respectively.
\item For each finite place $q \nmid pN^+$, the isomorphism
$i_q : B_{nN^-,q} \simeq \mathrm{M}_2(\mathbb{Q}_q)$ is chosen so that
$i_q \left( \mathcal{O}_K \otimes \mathbb{Z}_q  \right) \subseteq \mathrm{M}_2(\mathbb{Z}_q) $.
\end{itemize}
Under the fixed isomorphism $i$, for any rational prime $q$, the local Gross point $\varsigma_q \in B^\times_{nN^-,q}$ is defined as follows:
\begin{itemize}
\item $\varsigma_q := 1$
in $B^\times_{nN^-,q}$ for $q \nmid pN^+$.
\item $\varsigma_q := \frac{1}{\sqrt{D_K}}\cdot \left( \begin{matrix}
\vartheta & \overline{\vartheta} \\
1 & 1
\end{matrix} \right) \in \mathrm{GL}_2(K_\mathfrak{q}) = \mathrm{GL}_2(\mathbb{Q}_q) $
for $q \mid N^+$ with $q = \mathfrak{q} \overline{\mathfrak{q}}$ in $\mathcal{O}_K$.
\item 
$\varsigma^{(0)}_p = \left( \begin{matrix}
\vartheta & -1 \\
1 & 0
\end{matrix} \right)
 \in \mathrm{GL}_2(K_\mathfrak{p}) = \mathrm{GL}_2( \mathbb{Q}_{p} )$ when $p = \mathfrak{p}\overline{\mathfrak{p}}$ splits in $K$.
\item $\varsigma^{(0)}_p = \left( \begin{matrix}
0 & 1 \\
-1 & 0
\end{matrix} \right)
 \in \mathrm{GL}_2(K_p) = \mathrm{GL}_2( \mathbb{Q}_{p^2} )$ when $p$ is inert in $K$.
\end{itemize}
The \textbf{Gross point $\varsigma^{(0)}$ of conductor $1$ on $\widehat{B}^\times_{nN^-}$} is defined by
$$\varsigma^{(0)} := \varsigma^{(0)}_p \times \prod_{q \neq p} \varsigma_q \in \widehat{B}^\times_{nN^-} .$$
The action of $\widehat{K}^\times$ on $\varsigma^{(0)}$ is defined by left translation with embedding $K \hookrightarrow B$.

\bibliographystyle{amsalpha}
\bibliography{library}

\providecommand{\bysame}{\leavevmode\hbox to3em{\hrulefill}\thinspace}
\providecommand{\MR}{\relax\ifhmode\unskip\space\fi MR }
\providecommand{\MRhref}[2]{%
  \href{http://www.ams.org/mathscinet-getitem?mr=#1}{#2}
}
\providecommand{\href}[2]{#2}
\begin{thebibliography}{BCDT01}

\bibitem[BCDT01]{bcdt}
C.~Breuil, B.~Conrad, F.~Diamond, and R.~Taylor, \emph{On the modularity of
  elliptic curves over {$\mathbb{Q}$}: wild 3-adic exercises}, J. Amer. Math.
  Soc. \textbf{14} (2001), no.~4, 843--939.

\bibitem[BCGS]{burungale-castella-grossi-skinner-indivisibility}
A.~A. Burungale, F.~Castella, G.~Grossi, and C.~Skinner, \emph{Indivisibility
  of {K}olyvagin systems and {I}wasawa theory}, preprint.

\bibitem[BCK21]{burungale-castella-kim}
A.~A. Burungale, F.~Castella, and C.-H. Kim, \emph{A proof of {P}errin-{R}iou's
  {H}eegner point main conjecture}, Algebra Number Theory \textbf{15} (2021),
  no.~10, 1627--1653.

\bibitem[BD05]{bertolini-darmon-imc-2005}
M.~Bertolini and H.~Darmon, \emph{Iwasawa's main conjectures for elliptic
  curves over anticyclotomic {$\mathbb{Z}_p$}-extensions}, Ann. of Math. (2)
  \textbf{162} (2005), no.~1, 1--64.

\bibitem[BDV22]{bertolini-darmon-venerucci}
M.~Bertolini, H.~Darmon, and R.~Venerucci, \emph{Heegner points and
  {B}eilinson--{K}ato elements: a conjecture of {P}errin-{R}iou}, Adv. Math.
  \textbf{398} (2022), 108172.

\bibitem[BFH90]{bump-friedberg-hoffstein}
D.~Bump, S.~Friedberg, and J.~Hoffstein, \emph{Nonvanishing theorems for
  {$L$}-functions of modular forms and their derivatives}, Invent. Math.
  \textbf{102} (1990), no.~3, 543--618.

\bibitem[BPS]{kazim-pollack-sasaki}
K.~B{\"{u}}y{\"{u}}kboduk, R.~Pollack, and S.~Sasaki, \emph{{$p$}-adic
  {G}ross--{Z}agier formula at critical slope and a conjecture of
  {P}errin-{R}iou}, preprint,
  \href{https://arxiv.org/abs/1811.08216}{arXiv:1811.08216}.

\bibitem[BSTW]{burungale-skinner-tian-wan}
A.~A. Burungale, C.~Skinner, Y.~Tian, and X.~Wan, \emph{Zeta elements for
  elliptic curves and application}, in preparation.

\bibitem[BT20]{burungale-tian-p-converse}
A.~A. Burungale and Y.~Tian, \emph{{$p$}-converse to a theorem of
  {G}ross--{Z}agier, {K}olyvagin and {R}ubin}, Invent. Math. \textbf{220}
  (2020), no.~1, 211--253.

\bibitem[B{\"{u}}y09]{kazim-tamagawa}
K.~B{\"{u}}y{\"{u}}kboduk, \emph{Tamagawa defect of {E}uler systems}, J. Number
  Theory \textbf{129} (2009), no.~2, 402--417.

\bibitem[CH15]{chida-hsieh-main-conj}
M.~Chida and M.-L. Hsieh, \emph{On the anticyclotomic {I}wasawa main conjecture
  for modular forms}, Compos. Math. \textbf{151} (2015), no.~5, 863--897.

\bibitem[CH18]{chida-hsieh-p-adic-L-functions}
\bysame, \emph{Special values of anticyclotomic {$L$}-functions for modular
  forms}, J. Reine Angew. Math. \textbf{741} (2018), 87--131.

\bibitem[Cor02]{cornut-higher-heegner-points}
C.~Cornut, \emph{Mazur's conjecture on higher {H}eegner points}, Invent. Math.
  \textbf{148} (2002), no.~3, 495--523.

\bibitem[CW22]{castella-wan-imc-derivatives}
F.~Castella and X.~Wan, \emph{The {I}wasawa main conjectures for
  {$\mathrm{GL}_2$} and derivatives of {$p$}-adic {$L$}-functions}, Adv. Math.
  \textbf{400} (2022), no.~108266, 45 pages.

\bibitem[GZ86]{gross-zagier-original}
B.~Gross and D.~Zagier, \emph{Heegner points and derivatives of {$L$}-series},
  Invent. Math. \textbf{84} (1986), no.~2, 225--320.

\bibitem[How04]{howard-kolyvagin}
B.~Howard, \emph{The {H}eegner point {K}olyvagin system}, Compos. Math.
  \textbf{140} (2004), no.~6, 1439--1472.

\bibitem[How06]{howard-bipartite}
\bysame, \emph{Bipartite {E}uler systems}, J. Reine Angew. Math. \textbf{597}
  (2006), 1--25.

\bibitem[Jet08]{jetchev-global-divisibility}
D.~Jetchev, \emph{Global divisibility of {H}eegner points and {T}amagawa
  numbers}, Compositio Math. \textbf{144} (2008), 811--826.

\bibitem[JSW17]{jetchev-skinner-wan}
D.~Jetchev, C.~Skinner, and X.~Wan, \emph{The {B}irch and {S}winnerton-{D}yer
  formula for elliptic curves of analytic rank one}, Camb. J. Math. \textbf{5}
  (2017), no.~3, 369--434.

\bibitem[Kat04]{kato-euler-systems}
K.~Kato, \emph{{$p$}-adic {H}odge theory and values of zeta functions of
  modular forms}, Ast\'{e}risque \textbf{295} (2004), 117--290.

\bibitem[Kim]{kim-structure-selmer}
C.-H. Kim, \emph{The structure of {S}elmer groups and the {I}wasawa main
  conjecture for elliptic curves}, submitted,
  \href{https://arxiv.org/abs/2203.12159}{arXiv:2203.12159}.

\bibitem[Kim19]{kim-overconvergent}
\bysame, \emph{Overconvergent quaternionic forms and anticyclotomic {$p$}-adic
  {$L$}-functions}, Publ. Mat. \textbf{63} (2019), no.~2, 727--767.

\bibitem[Kol90]{kolyvagin-euler-systems}
V.~Kolyvagin, \emph{Euler systems}, The {G}rothendieck {F}estschrift {V}olume
  {II} (P.~Cartier, L.~Illusie, N.~M. Katz, G.~Laumon, Y.~Manin, and K.~A.
  Ribet, eds.), Progr. Math., vol.~87, Birkh\"{a}user {B}oston, 1990,
  pp.~435--483.

\bibitem[Kol91]{kolyvagin-selmer}
\bysame, \emph{On the structure of {S}elmer groups}, Math. Ann. \textbf{291}
  (1991), no.~2, 253--259.

\bibitem[KPW17]{kim-pollack-weston}
C.-H. Kim, R.~Pollack, and T.~Weston, \emph{On the freeness of anticyclotomic
  {S}elmer groups of modular forms}, Int. J. Number Theory \textbf{13} (2017),
  no.~6, 1443--1455.

\bibitem[Kur14a]{kurihara-munster}
M.~Kurihara, \emph{Refined {I}wasawa theory for {$p$}-adic representations and
  the structure of {S}elmer groups}, M{\"{u}}nster J. of Math. \textbf{7}
  (2014), no.~1, 149--223.

\bibitem[Kur14b]{kurihara-iwasawa-2012}
\bysame, \emph{The structure of {S}elmer groups of elliptic curves and modular
  symbols}, Iwasawa Theory 2012: State of the Art and Recent Advances
  (T.~Bouganis and O.~Venjakob, eds.), Contrib. Math. Comput. Sci., vol.~7,
  Springer, 2014, pp.~317--356.

\bibitem[Maz78]{mazur-rational-isogenies}
B.~Mazur, \emph{Rational isogenies of prime degree}, Invent. Math. \textbf{44}
  (1978), 129--162, with an appendix by D. Goldfeld.

\bibitem[MM91]{murty-murty-mean}
M.~R. Murty and V.~K. Murty, \emph{Mean values of derivatives of modular
  {$L$}-series}, Ann. of Math. (2) \textbf{133} (1991), no.~3, 447--475.

\bibitem[MR04]{mazur-rubin-book}
B.~Mazur and K.~Rubin, \emph{{K}olyvagin {S}ystems}, Mem. Amer. Math. Soc.,
  vol. 168, American {M}athematical {S}ociety, March 2004.

\bibitem[MT87]{mazur-tate}
B.~Mazur and J.~Tate, \emph{Refined conjectures of the ``{B}irch and
  {S}winnerton-{D}yer type"}, Duke Math. J. \textbf{54} (1987), no.~2,
  711--750.

\bibitem[Nek07]{nekovar-euler-systems}
J.~Nekov\'{a}\v{r}, \emph{The {E}uler system method for {CM} points on
  {S}himura curves}, {$L$}-functions and {G}alois representations (Cambridge)
  (D.~Burns, K.~Buzzard, and J.~Nekov\'{a}\v{r}, eds.), London Math. Soc.
  Lecture Note Ser., vol. 320, Cambridge University Press, 2007, pp.~471--547.

\bibitem[PR87]{perrin-riou-heegner}
B.~Perrin-Riou, \emph{Fonctions {$L$} {$p$}-adiques, th\'{e}orie d'{I}wasawa et
  points de {H}eegner}, Bull. Soc. Math. France \textbf{115} (1987), no.~4,
  399–--456.

\bibitem[PR93]{perrin-riou-rational-pts}
\bysame, \emph{Fonctions {$L$} {$p$}-adiques d'une courbe elliptique et points
  rationnels}, Ann. Inst. Fourier (Grenoble) \textbf{43} (1993), no.~4,
  945--995.

\bibitem[PW11]{pw-mu}
R.~Pollack and T.~Weston, \emph{On anticyclotomic {$\mu$}-invariants of modular
  forms}, Compos. Math. \textbf{147} (2011), 1353--1381.

\bibitem[RS07]{rubin-silverberg-twists-rank-four}
K.~Rubin and A.~Silverberg, \emph{Twists of elliptic curves of rank at least
  four}, Ranks of elliptic curves and random matrix theory (Cambridge) (J.B.
  Conrey, D.W. Farmer, F.~Mezzadri, and N.C. Snaith, eds.), London Math. Soc.
  Lecture Note Ser., vol. 341, 2007, pp.~177--188.

\bibitem[Sak22]{sakamoto-p-selmer}
R.~Sakamoto, \emph{$p$-{S}elmer groups and modular symbols}, Doc. Math.
  \textbf{27} (2022), 1891--1922.

\bibitem[SU14]{skinner-urban}
C.~Skinner and E.~Urban, \emph{The {I}wasawa main conjectures for
  {$\mathrm{GL}_2$}}, Invent. Math. \textbf{195} (2014), no.~1, 1--277.

\bibitem[SW]{stein-weinstein}
W.~Stein and J.~Weinstein, \emph{Kolyvagin classes on elliptic curves:
  structure, distribution, and algorithms}, unpublished manuscript.

\bibitem[Swe]{sweeting-kolyvagin}
N.~Sweeting, \emph{Kolyvagin's conjecture, bipartite {E}uler systems and higher
  congruences of modular forms}, preprint available at
  \href{https://scholar.harvard.edu/naomisweeting}{the author's webpapge} and
  \href{https://arxiv.org/abs/2012.11771}{arXiv:2012.11771}.

\bibitem[Vat02]{vatsal-uniform}
V.~Vatsal, \emph{Uniform distribution of {H}eegner points}, Invent. Math.
  \textbf{148} (2002), no.~1, 1--48.

\bibitem[Vat03]{vatsal-duke}
\bysame, \emph{Special values of anticyclotomic {$L$}-functions}, Duke Math. J.
  \textbf{116} (2003), no.~2, 219--261.

\bibitem[Vig80]{vigneras}
M.-{F}. Vign\'{e}ras, \emph{Arithm\'{e}tique des alg\`{e}bres de quaternions},
  Lecture Notes in Math., vol. 800, Springer, 1980.

\bibitem[Wal85]{waldspurger}
J.-L. Waldspurger, \emph{Sur les values de certaines fonctions {$L$}
  automorphes en leur centre de symm\'{e}tre}, Compos. Math. \textbf{54}
  (1985), 173--242.

\bibitem[Wan21]{wan-heegner}
X.~Wan, \emph{Heegner point {K}olyvagin system and {I}wasawa main conjecture},
  Acta Math. Sin. (Engl. Ser.) \textbf{37} (2021), no.~1, 104--120.

\bibitem[YZ17]{yun-zhang-higher-gross-zagier}
Z.~Yun and W.~Zhang, \emph{Shtukas and the {T}aylor expansion of
  {$L$}-functions}, Ann. of Math. \textbf{186} (2017), no.~3, 767--911.

\bibitem[YZ19]{yun-zhang-higher-gross-zagier-2}
\bysame, \emph{Shtukas and the {T}aylor expansion of {$L$}-functions {(II)}},
  Ann. of Math. \textbf{189} (2019), no.~2, 393--526.

\bibitem[YZZ13]{yuan-zhang-zhang}
X.~Yuan, S.-W. Zhang, and W.~Zhang, \emph{The {G}ross--{Z}agier formula on
  {S}himura curves}, Ann. of Math. Stud., vol. 184, Princeton {U}niversity
  {P}ress, 2013.

\bibitem[Zan]{zanarella-howard}
M.~Zanarella, \emph{On {H}oward's main conjecture and the {H}eegner point
  {K}olyvagin system}, preprint,
  \href{https://arxiv.org/abs/1908.09197}{arXiv:1908.09197}.

\bibitem[Zha12]{wei-zhang-arith-fund-lem}
W.~Zhang, \emph{On arithmetic fundamental lemmas}, Invent. Math. \textbf{188}
  (2012), 197--252.

\bibitem[Zha14a]{wei-zhang-cdm}
\bysame, \emph{The {B}irch--{S}winnerton-{D}yer conjecture and {H}eegner
  points: a survey}, Current {D}evelopments in {M}athematics, vol. 2013, 2014,
  pp.~169--203.

\bibitem[Zha14b]{wei-zhang-mazur-tate}
\bysame, \emph{Selmer groups and the indivisibility of {H}eegner points}, Camb.
  J. Math. \textbf{2} (2014), no.~2, 191--253.

\bibitem[Zha21]{wei-zhang-weil-repns-afl}
\bysame, \emph{Weil representation and arithmetic fundamental lemma}, Ann. of
  Math. (2) \textbf{193} (2021), no.~3, 863--978.

\end{thebibliography}

\end{document}